\documentclass{dcds}
\usepackage{amsmath}
  \usepackage{paralist}
    \usepackage{amssymb}
      \usepackage{graphics} 
  \usepackage{epsfig}
  \usepackage{amscd,amsxtra}
\usepackage{dsfont}
\usepackage{url}
\usepackage{mathrsfs}
 \usepackage[colorlinks=true]{hyperref}
\hypersetup{urlcolor=blue, citecolor=red}

\def\currentvolume{33}
 \def\currentissue{8}
  \def\currentyear{2013}
   \def\currentmonth{August}
    \def\ppages{X--XX}
     \def\DOI{10.3934/dcds.2013.33.xx}

\renewcommand{\mathcal}{\mathscr}

\definecolor{citation}{rgb}{0.2,0.5,0.2}
\definecolor{formula}{rgb}{0.1,0.2,0.5}
\definecolor{url}{rgb}{0,0.2,0.7}

\newtheorem{theorem}{Theorem}[section]
\newtheorem{corollary}[theorem]{Corollary}
\newtheorem{lemma}[theorem]{Lemma}
\newtheorem{prop}[theorem]{Proposition}

\theoremstyle{definition}
\newtheorem{defn}[theorem]{Definition}
\newtheorem{rem}[theorem]{Remark}
\numberwithin{equation}{section}

\newcommand{\R}{{\mathds R}}

\newlength{\defbaselineskip}
\title
[Geometric inequalities and symmetry results for elliptic systems]
{Geometric inequalities and symmetry \\
results for elliptic systems}

\author[Serena Dipierro]{}


\email{dipierro@sissa.it}

\thanks{The author has been supported by FIRB ``Project Analysis and Beyond."}

\begin{document}
\maketitle

\centerline{\scshape Serena Dipierro }
\medskip
{\footnotesize
 \centerline{SISSA - International School for Advanced Studies}
   \centerline{Sector of Mathematical Analysis}
   \centerline{Via Bonomea, 265}
   \centerline{34136 Trieste, Italy}
}
\bigskip

\centerline{(Communicated by Alessio Figalli)}

\begin{abstract}
We obtain some Poincar\'{e} type formulas, that we use, together with the level set analysis,
to detect the one-dimensional symmetry of monotone and stable solutions of
possibly degenerate elliptic systems of the form
\begin{eqnarray*}
\left\{
\begin{array}{ll}
    div\left( a\left( |\nabla u|\right) \nabla u\right)   = F_1(u, v),         \\
    div\left( b\left( |\nabla v|\right) \nabla v\right) = F_2(u, v),
 \end{array}
\right.
\end{eqnarray*}
where~$F\in C^{1,1}_{loc}(\R^2)$.

Our setting is very general, and it comprises, as a particular case,
a conjecture of De Giorgi for phase separations in~$\R^2$.
\end{abstract}

\section{Introduction}
In this paper we consider a class of quasilinear (possibly degenerate) elliptic systems in $\R^n$.
We prove that, under suitable assumptions, the solutions have one-dimensional symmetry,
showing that the results obtained in \cite{BLWZ, BSWW, FG} hold in a more general setting.

In \cite{BLWZ} the following problem has been studied:
\begin{eqnarray}\label{syst}
\left\{
\begin{array}{ll}
    \Delta u   = uv^2,         \\
    \Delta v = vu^2, \\
    u, v>0.
 \end{array}
\right.
\end{eqnarray}
The authors proved the existence, symmetry and nondegeneracy of the solution to problem \eqref{syst}
in $\R$; in particular, they showed that entire solutions are reflectionally symmetric,
namely that there exists $x_0$ such that $u(x-x_0)=v(x-x_0)$.
Moreover, they estabilished a result that may be considered the analogue of a famous
conjecture of De Giorgi for problem \eqref{syst} in dimension $2$,
that is they proved that monotone solutions of \eqref{syst} in $\R^2$ have one-dimensional symmetry
under the additional growth condition
\begin{equation}\label{growth}
  u(x)+v(x) \leq C(1+|x|).
\end{equation}
On the other hand, in \cite{NTTV}, it has been proved that
the linear growth is the lowest possible for solutions to \eqref{syst};
in other words, if there exists $\alpha\in (0,1)$ such that
$$
   u(x)+v(x) \leq C(1+|x|)^{\alpha},
$$
then $u=v\equiv 0$.

In \cite{BSWW} the authors proved that the above mentioned one-dimensional symmetry still holds
in $\R^2$ when the monotonicity condition is replaced by the stability of the solutions
(which is a weaker assumption).
Moreover, they showed that there exist solutions to \eqref{syst}
which do not satisfy the growth condition \eqref{growth},
by constructing solutions with polynomial growth.

Moreover, we mention the paper \cite{Wa},
where the author proved that, for any $n\geq 2$,
a solution to \eqref{syst} which is a local minimizer and satisfies the growth condition \eqref{growth}
has one-dimensional symmetry.


In this paper we consider a more general setting,
that is we take $F\in C^{1,1}_{loc}\left(\R^2\right)$,
and we study the following elliptic system in~$\R^n$
\begin{eqnarray}\label{system}
\left\{
\begin{array}{ll}
    div\left( a\left( |\nabla u|\right) \nabla u\right)   = F_1(u, v),         \\
    div\left( b\left( |\nabla v|\right) \nabla v\right) = F_2(u, v),
 \end{array}
\right.
\end{eqnarray}
where $F_1$ and $F_2$ denote the derivatives of $F$ with respect to the first and the second variable respectively.

We suppose that $a, b\in C^1\left( \left( 0, +\infty\right) \right)$ satisfy
the following conditions:
\begin{equation}\label{ab1}
a(t)> 0, \quad b(t)> 0 \quad \mathrm{for\ any\ } t\in (0, +\infty),
\end{equation}
\begin{equation}\label{ab2}
a(t) + a'(t)t> 0, \quad b(t)+b'(t)t>0 \quad \mathrm{for\ any\ } t\in (0, +\infty).
\end{equation}

We define $A, B : \R^n \setminus \left\lbrace 0\right\rbrace \rightarrow Mat\left(n\times n\right)$
by setting, for any $1\leq h, k \leq n$,
$$
A_{hk}(\xi) := \frac{a'\left(|\xi |\right)}{|\xi |} \, \xi_h \, \xi_k + a(|\xi |) \delta_{hk},
$$
and
$$
B_{hk} (\xi) := \frac{b'\left(|\xi |\right)}{|\xi |} \, \xi_h \, \xi_k + b(|\xi |) \delta_{hk}.
$$

Now, for any $t>0$, we introduce the following notation:
\begin{equation}\label{lambda}
\lambda_1 (t) := a(t) + a'(t)t, \qquad     \lambda_2 (t) = \ldots = \lambda_n (t) := a(t),
\end{equation}
\begin{equation}\label{gamma}
\gamma_1 (t) := b(t) + b'(t)t, \qquad    \gamma_2 (t) = \ldots = \gamma_n (t) := b(t),
\end{equation}
and we define
$$
\Lambda_i (t) := \int_0^t \lambda_i \left( |s|\right) s\, ds,   \qquad
\Gamma_i (t) := \int_0^t \gamma_i \left( |s|\right) s\, ds
$$
for $i=1,2$ and $t\in\R$.

We will require that $a$ satisfies (A1) or (A2), where:
\begin{itemize}
\item[(A1)] $\left\lbrace \nabla u =0\right\rbrace =\varnothing$ and
$$
t^2 \lambda_1 (t) \in L^{\infty}_{loc}\left( \left[ 0, +\infty\right) \right).
$$
\item[(A2)] We have that
$$
a\in C\left( \left[ 0, +\infty\right) \right)
$$
and
$$
\mathrm{the\ map\ } t \mapsto ta(t) \mathrm{\ belongs\ to\ } C^1 \left( \left[ 0, +\infty\right) \right).
$$
\end{itemize}
Moreover, we require the same properties for $b$:
\begin{itemize}
\item[(B1)] $\left\lbrace \nabla v =0\right\rbrace =\varnothing$ and
$$
t^2 \gamma_1 (t) \in L^{\infty}_{loc}\left( \left[ 0, +\infty\right) \right).
$$
\item[(B2)] We have that
$$
b\in C\left( \left[ 0, +\infty\right) \right)
$$
and
$$
\mathrm{the\ map\ } t \mapsto tb(t) \mathrm{\ belongs\ to\ } C^1 \left( \left[ 0, +\infty\right) \right).
$$
\end{itemize}
In case (A2) and (B2) hold, we define
$A_{hk}(0) := a(0) \delta_{hk}$ and $B_{hk}(0) := b(0) \delta_{hk}$.

These assumptions may look rather technical at a first glance,
but they are the standard conditions that comprise as particular cases
the classical elliptic degenerate and nonlinear operators,
such as the $p$-Laplacian and the mean curvature operator.

In order to state our main result, we give the definition of monotone and stable solution.
\begin{defn}\label{monsol}
We say that a solution~$(u, v)$ of~\eqref{system} satisfies
a \emph{monotonicity condition} if
\begin{equation}\label{monotonicity}
 u_n >0, \qquad     v_n<0.
\end{equation}
\end{defn}

\begin{defn}\label{stablesol}
When $F\in C^2_{loc}(\R^2)$ we say that a solution~$(u, v)$ of~\eqref{system}
is \emph{stable} if the linearization is weakly positive definite,
that is, for any~$\phi, \psi\in C^{\infty}_0(\R^n)$,
\begin{equation}\begin{split}\label{stable}
& \int_{\R^n} \left( A\left( \nabla u(x)\right) \nabla \phi (x)\right) \cdot \nabla\phi (x)  +
   \left( B\left( \nabla v(x)\right) \nabla\psi(x)\right) \cdot \nabla\psi (x)  \\
&\qquad +    F_{11}(u,v) \phi^2(x) + F_{22}(u,v) \psi^2(x) + 2 F_{12}(u,v) \phi(x)\psi (x) \, dx \geq 0. \end{split}
\end{equation}
\end{defn}

In our general framework, since $F_1$ and $F_2$ may not be everywhere differentiable,
the integral in \eqref{stable} may not be well defined.
Therefore it is convenient to introduce the sets
$$
  \mathcal{D} :=  \left\lbrace (t,s)\in\R^2 : F_{11}(t,s),\, F_{12}(t,s),\,
 F_{22}(t,s) \mathrm{\ exist} \right\rbrace ,
$$
and
$$
  \mathcal{N} := \R^2 \setminus\mathcal{D}.
$$
It is known that
\begin{equation}\label{borel}
\mathrm{the\ set\ } \mathcal{N} \mathrm{\ is\ Borel\ and\ with\ zero\ Lebesgue\ measure\ }
\end{equation}
(see pages 81--82 in \cite{EG}).
Moreover, we consider the sets
\begin{eqnarray*}&&
  \mathcal{N}_{uv} := \left\lbrace x\in\R^n : (u(x), v(x))\in\mathcal{N} \right\rbrace,
\\{\mbox{and
}}&&
  \mathcal{D}_{uv} := \R^n \setminus\mathcal{N}_{uv}.
\end{eqnarray*}
So we say that $(u,v)$ is a stable solution to \eqref{system},
if for any $\phi, \psi\in C^{\infty}_0(\R^n)$,
\begin{equation}\begin{split}\label{stable1}
 & \int_{\R^n} \left( A\left( \nabla u(x)\right) \nabla \phi (x)\right) \cdot \nabla\phi (x)  +
   \left( B\left( \nabla v(x)\right) \nabla\psi(x)\right) \cdot \nabla\psi (x)\, dx  \\
  &\qquad +\int_{\mathcal{D}_{uv}}  F_{11}(u,v) \phi^2(x) + F_{22}(u,v) \psi^2(x) +
2F_{12}(u,v) \phi(x)\psi (x) \, dx \geq 0.
\end{split}\end{equation}
Of course, \eqref{stable1} reduces to \eqref{stable} when $F$ is in $C^2_{loc}(\R^2)$.

Then, we state our symmetry result. For this, we denote by~$\Im (u,v)$ the 
image of the map~$(u,v):\R^n\rightarrow\R^2$, i.e. 
$\Im (u,v):=\left\lbrace (u(x),v(x)), x\in\R^n\right\rbrace$.  
\begin{theorem}\label{T1D}
Let~$(u, v)$ be a solution of~\eqref{system}. Suppose that~$u\in
C^1(\R^n)\cap
C^2(\left\lbrace \nabla u\neq 0\right\rbrace)$,~$v\in C^1(\R^n)\cap C^2(\left\lbrace \nabla v\neq 0\right\rbrace)$, and~$\nabla u, \nabla v\in W^{1,2}_{loc}(\R^n)$.

Suppose that either (A2) holds or that~$\left\lbrace \nabla u=0\right\rbrace =\varnothing$, and that either (B2) holds or that~$\left\lbrace \nabla v=0\right\rbrace=\varnothing$.

Assume that either
\begin{equation}\label{mon}
\mbox{the\ monotonicity\ condition\ \eqref{monotonicity}\ holds,\ and\ $F_{12}(u,v)\geq 0$\ in\ $\Im(u,v)$}  , \end{equation}
or
\begin{equation}\label{stab}
  \mbox{$(u,v)$\ is\ stable,\ and\ $F_{12}(u,v)\leq 0$\ in\ $\Im(u,v)$}  . \end{equation}
If
\begin{equation}\label{EE}
\liminf_{R\rightarrow+\infty}\frac1{\log^2 R}
\int_{B_R\setminus B_{\sqrt R}}
\frac{|A(\nabla u(x))|\, |\nabla u(x)|^2 +|B(\nabla v(x))|\, |\nabla v(x)|^2}{|x|^2}\, dx =0,
\end{equation}
then~$(u, v)$ has one-dimensional symmetry, in the sense that there exist~$\overline{u}, \overline{v}:\R\rightarrow\R$ and~$\omega_{u}, \omega_{v}\in S^{n-1}$ in such a way that~
$(u(x), v(x))=(\overline{u}(\omega_{u}\cdot x), \overline{v}(\omega_{v}\cdot x))$, for any~$x\in\R^n$.

Moreover, if we assume in addition that either
\begin{equation}\begin{split}\label{monF12}
&\mbox{the\ monotonicity\ condition\ \eqref{monotonicity}\ holds,\ and\ there\ exists\ a\ non-empty}\\
&\mbox{open\ set\ $\Omega'\subseteq\R^n$\ such\ that\ $F_{12}(u(x),v(x))>0$\ for\ any\ $x\in\Omega'$}  , 
\end{split}\end{equation}
or
\begin{equation}\begin{split}\label{stabF12}
&\mbox{$(u,v)$\ is\ stable,\ and\ there\ exist\ two\ open\ intervals\ $I_u,I_v\subseteq\R$}\\
&\mbox{such\ that\ $\left(I_u\times I_v\right)\cap\Im(u,v)\neq\varnothing$\ and\ $F_{12}(\overline u,\overline v)>0$\ for\ any\ $(\overline u,\overline v)\in I_u\times I_v$,}
\end{split}\end{equation}
then~$(u, v)$ has one-dimensional symmetry, and $\omega_{u}=\omega_{v}$.
\end{theorem}

\begin{rem}
Notice that the hypothesis that~$F_{12}(u,v)$ is not identically zero cannot be removed
if we want to conclude that~$u,v$ have one-dimensional symmetry with the same
unit vector~$\omega$.
Indeed, in~$\R^2$ one can consider the system in \eqref{system} with the Laplace
operator and~$F\equiv0$.
Then, if one take the functions~$u(x_1,x_2)=x_2$ and~$v(x_1,x_2)=x_1-x_2$,
it is easy to see that~$(u,v)$ is a monotone and stable solution to \eqref{system}
and \eqref{mon}, \eqref{stab} and \eqref{EE} are satisfied, but~$u$ and~$v$ have one-dimensional
symmetry with a different vector~$\omega$.

\noindent
Notice also that one can consider a more general function $F$ such
that~$F_{12}(u,v)=0$, that is a system with two independent equations,
and there is no reason why~$u$ and~$v$ should have one-dimensional symmetry
with the same vector.
\end{rem}

We notice that, as paradigmatic examples satisfying the assumptions of Theorem \ref{T1D},
one may take the $p$-Laplacian, with $p\in (1, +\infty)$ if $\left\lbrace \nabla u=0\right\rbrace=\varnothing$
and any $p\in [2, +\infty )$ if $\left\lbrace \nabla u=0\right\rbrace \neq\varnothing$
(in this case, for instance, $a(t)=t^{p-2}$)
or the mean curvature operator (in this case, $a(t)=(1+t^2)^{-1/2}$).
Moreover, we observe that Theorem \ref{T1D} holds even if $a$ and $b$ are two
different functions satisfying the hypotheses (e.g., one can take $a$ to be of $p$-Laplacian
type and $b$ of mean curvature type).

To prove Theorem \ref{T1D} we borrow a large number of ideas from \cite{Fa} and \cite{FSV},
and exploit some techniques of \cite{SZ1, SZ2}.
In particular, we will show that a formula proved in \cite{SZ1, SZ2} and its extension obtained in \cite{FSV}
for elliptic equations still hold for systems (see Corollaries \ref{cor1} and \ref{cor2}).
Since this formula bounds a weighted $L^2$-norm of any test function
by a weighted $L^2$-norm of its gradient,
we may see it as a weighted Poincar\'{e} type inequality.
Such a formula is geometric in spirit, since it bounds
tangential gradients and curvatures of level sets of monotone and stable solutions
in terms of suitable energy integrals.

Our result extends the one obtained in \cite{FG},
where the authors studied problem \eqref{system} in the case
$a=b=Id$, and use this kind of geometric Poincar\'{e} inequality to show
that in $\R^2$ any stable solution has a one-dimensional symmetry.
Of course in our setting several technical and conceptual complications
arise due to the possible degeneracy of the operators considered
and to the nonlinear dependence on the gradient terms.

Moreover, as a particular case, Theorem \ref{T1D} comprises a conjecture of De Giorgi
for phase separations in $\R^2$ (see the end of Section \ref{sec:appl}).

We refer the reader to \cite{FV} for a recent review on the conjecture of De Giorgi
and related topics.


The paper is organized as follows.
In Section \ref{sec:useful} we collect some preliminary material.
Sections \ref{sec:mon} and \ref{sec:stable} are devoted to show that
some geometric Poincar\'{e} type inequalities hold for monotone
and stable solutions to \eqref{system} respectively.
In Section \ref{sec:levset} we develop the level set analysis.
In Section \ref{sec:proof} we provide the proof of Theorem \ref{T1D},
by using the results obtained in the previous sections.
Finally, in Section \ref{sec:appl}, we give an application of Theorem \ref{T1D},
namely we prove that a conjecture of De Giorgi holds in $\R^2$
for systems like \eqref{system}, and in particular for phase separations.

\section{Some useful results}\label{sec:useful}
In this section we collect some results that we will use in the sequel.

First, we have the following lemma (see Lemma 2.1 in \cite{FSV} for a simple proof):
\begin{lemma}\label{lemma1}
For any $\xi\in\R^n \setminus \left\lbrace 0\right\rbrace$,
the matrices $A(\xi)$, $B(\xi)$ are symmetric and positive definite,
and their eigenvalues are $\lambda_1 (|\xi |), \ldots , \lambda_n (|\xi |)$
and $\gamma_1 (|\xi |), \ldots ,\gamma_n (|\xi |)$ respectively.

Moreover
$$
A(\xi) \xi \cdot \xi = |\xi |^2 \lambda_1 (|\xi |), \qquad
B(\xi) \xi \cdot \xi = |\xi |^2 \gamma_1 (|\xi |).
$$
\end{lemma}
It follows from Lemma \ref{lemma1} that, for any $t\in\R\setminus \left\lbrace 0\right\rbrace$,
$$
   \Lambda_i (-t) = \Lambda_i (t) > 0,    \qquad    \Gamma_i (-t) = \Gamma_i (t) > 0 .
$$
Moreover, for any $V, W \in\R^n$, and any $\xi\in\R^n \setminus \left\lbrace 0\right\rbrace$,
\begin{equation}\label{Axi}
     0 \leq A(\xi)\left(V-W\right) \cdot\left(V-W\right)   =
      A(\xi)V\cdot V +  A(\xi) W \cdot W    - 2 A(\xi) V\cdot W ,
\end{equation}
\begin{equation}\label{Bxi}
    0 \leq B(\xi)\left(V-W\right) \cdot\left(V-W\right)   =
      B(\xi)V\cdot V +  B(\xi) W \cdot W    - 2 B(\xi) V\cdot W .
\end{equation}

\begin{lemma}\label{lemma2}
Let $(u, v)$ be a weak solution of (\ref{system})
such that $u\in C^1 (\R^n) \cap C^2 (\left\lbrace \nabla u \neq 0\right\rbrace )$,
$v\in C^1 (\R^n) \cap C^2 (\left\lbrace \nabla v \neq 0\right\rbrace )$,
with $\nabla u, \nabla v \in W^{1,2}_{loc}(\R^n)$.
Suppose that either (A2) holds or that $\left\lbrace \nabla u = 0 \right\rbrace =\varnothing$,
and that either (B2) holds or that $\left\lbrace \nabla v = 0 \right\rbrace =\varnothing$.

Then, for any $j=1, \ldots, n$, $(u_j , v_j )$ is a weak solution of
\begin{eqnarray}\label{system1}
\left\{
\begin{array}{ll}
    div\left( A\left(\nabla u\right) \nabla u_j \right)   = F_{11}(u, v) u_j + F_{12} (u,v) v_j,         \\
    div\left( B\left(\nabla v\right) \nabla v_j \right) = F_{21}(u, v) u_j   +  F_{22}(u,v) v_j.
 \end{array}
\right.
\end{eqnarray}
\end{lemma}
\begin{proof}
First of all, we observe that
\begin{equation}\label{map}
\mathrm{the\ map\ } x\mapsto \mathcal{A}(x) := a ( |\nabla u(x) |) \nabla u(x) \mathrm{\ belongs\ to\ }
W^{1,1}_{loc}(\R^n , \R^n),
\end{equation}
and
\begin{equation}\label{map1}
\mathrm{the\ map\ } x\mapsto \mathcal{B}(x) := b ( |\nabla u(x) |) \nabla u(x) \mathrm{\ belongs\ to\ }
W^{1,1}_{loc}(\R^n , \R^n).
\end{equation}
Let us show \eqref{map}. It is obvious if $\left\lbrace \nabla u =0\right\rbrace =\varnothing$,
while, if (A2) holds, we have that the map
$$ \xi\in\R^n \mapsto \bar{\mathcal{A}} (\xi) := a(|\xi |) \xi $$
belongs to $W^{1, \infty}_{loc}(\R^n)$,
and so \eqref{map} follows by writing $\mathcal{A}(x) = \bar{\mathcal{A}} (\nabla u(x))$.
In the same way one shows \eqref{map1}.

From \eqref{map} and \eqref{map1}, we have that, for any $\phi,\psi\in C^{\infty}_0 (\R^n, \R^n)$,
$$ - \int_{\R^n} \partial_j \left( a\left( |\nabla u |\right) \nabla u\right) \cdot \phi \, dx =
    \int_{\R^n}  a\left( |\nabla u |\right) \nabla u \cdot \partial_j \phi \, dx, $$
and
$$ - \int_{\R^n} \partial_j \left( b\left( |\nabla v |\right) \nabla v\right) \cdot \psi \, dx =
    \int_{\R^n}  b\left( |\nabla v |\right) \nabla v \cdot \partial_j \psi \, dx. $$

Moreover, by \eqref{map} and Stampacchia's Theorem
(see, for instance, Theorem 6.19 of \cite{LL}), we get
that $\partial_j \mathcal{A}(x)=0$ for almost any $x\in\left\lbrace \mathcal{A}=0\right\rbrace$,
that is
$$ \partial_j \left(a\left( |\nabla u(x) |\right) \nabla u(x)\right) =0 $$
for almost any $x\in \left\lbrace \nabla u =0\right\rbrace$.

Similarly, by using again Stampacchia's Theorem and (A2),
we conclude
that $\nabla u_j (x)=0$, and then $A\left(\nabla u(x)\right)
\nabla u_j (x)=0$,
for almost any $x\in\left\lbrace \nabla u =0\right\rbrace$.

A direct computation also shows that on $\left\lbrace \nabla u \neq 0\right\rbrace$
$$ \partial_j \left( a\left( |\nabla u |\right) \nabla u\right) = A\left( \nabla u\right) \nabla u_j .$$

As a consequence,
$$ \partial_j \left( a\left( |\nabla u |\right) \nabla u\right) = A\left( \nabla u\right) \nabla u_j $$
almost everywhere.

Reasoning in the same way, we conclude also that
$$ \partial_j \left( b\left( |\nabla v |\right) \nabla v\right) = B\left( \nabla v\right) \nabla v_j $$
almost everywhere.

Let now $\phi, \psi\in C^{\infty}_0 \left( \R^n\right)$.
We use the above observations to obtain that
\begin{eqnarray*}
 && - \int_{\R^n} A\left(\nabla u\right) \nabla u_j \cdot \nabla\phi
     + F_{11}(u,v) u_j \phi   +   F_{12}(u,v) v_j \phi  \, dx   \\
  & =& -  \int_{\R^n} \partial_j \left( a\left( |\nabla u|\right) \nabla u \right) \cdot \nabla\phi
      + \partial_j \left( F_1 (u,v)\right) \phi  \, dx \\
  &=& \int_{\R^n} a\left( |\nabla u|\right) \nabla u \cdot \nabla \phi_j
          +    F_1 (u,v) \phi_j \, dx,
\end{eqnarray*}
and
\begin{eqnarray*}
  &&- \int_{\R^n} B\left(\nabla v\right) \nabla v_j \cdot \nabla\psi
     + F_{21}(u,v) u_j \psi   +   F_{22}(u,v) v_j \psi  \, dx   \\
    &=& -  \int_{\R^n} \partial_j \left( b\left( |\nabla v|\right) \nabla v \right) \cdot \nabla\psi
      + \partial_j \left( F_2 (u,v)\right) \psi  \, dx \\
    &=& \int_{\R^n} b\left( |\nabla v|\right) \nabla v \cdot \nabla \psi_j
          +    F_2 (u,v) \psi_j \, dx,
\end{eqnarray*}
which vanish, since $(u,v)$ is a weak solution of \eqref{system}.
\end{proof}

We observe that in the proof of Lemma \ref{lemma2} it is sufficient to assume
that $\nabla u, \nabla v\in W^{1,1}_{loc}(\R^n)$.
Since such a generality is not needed here, we assumed,
for simplicity, $\nabla u, \nabla v\in W^{1,2}_{loc}(\R^n)$
in order to use the above result in the sequel.

Let us notice that \eqref{system1} means that, for any
$\phi, \psi\in C^{\infty}_0 (\R^n)$, and for any~$j=1, \ldots , n$,
\begin{equation}\begin{split}\label{system2}
 & \int_{\R^n} A\left(\nabla u\right)\nabla u_j \cdot \nabla\phi + F_{11}(u,v)\, u_j\,\phi + F_{12}(u,v) v_j \phi \, dx =0, \\
& \int_{\R^n} B\left(\nabla v\right)\nabla v_j \cdot \nabla\psi + F_{12}(u,v)\, u_j \, \psi + F_{22}(u,v)\, v_j \, \psi \, dx =0.
                \end{split}
\end{equation}
Since the integrals in \eqref{system2} may not be well defined,
recalling the definitions of the sets $\mathcal{D}, \mathcal{N}, \mathcal{N}_{uv}, \mathcal{D}_{uv}$
given in the Introduction and using \eqref{borel} we can say that~$(u_j , v_j)$ satisfies
\begin{equation}\begin{split}\label{system3}
 & \int_{\R^n} A\left(\nabla u\right)\nabla u_j \cdot \nabla\phi \, dx
+ \int_{\mathcal{D}_{uv}} F_{11}(u,v)\, u_j\,\phi + F_{12}(u,v) v_j \phi \, dx =0, \\
& \int_{\R^n} B\left(\nabla v\right)\nabla v_j \cdot \nabla\psi \, dx
+ \int_{\mathcal{D}_{uv}}F_{12}(u,v)\, u_j \, \psi + F_{22}(u,v)\, v_j \, \psi \, dx =0.
                \end{split}
\end{equation}

In the sequel we will need to use \eqref{system3} for a less regular test functions.
To do this, we prove the following:
\begin{lemma}\label{lemma2bis}
Under the assumptions of Lemma \ref{lemma2}, we have that \eqref{system3}
holds for any~$j=1, \ldots, n$, any~$\phi, \psi\in W^{1,2}_0 (B)$ and any ball~$B \subset\R^n$.
\end{lemma}

\begin{proof}
Let us prove the first equality in~\eqref{system3}.
Given~$\phi\in W^{1,2}_0 (B)$, we consider a sequence of functions~$\phi_k \in C^{\infty}_0 (B)$ which converge to~$\phi$ in~$W^{1,2}_0 (B)$.
Let~$m_u$ and~$M_u$ (respectively~$m_v$ and~$M_v$) be the minimun and the
maximum of~$|\nabla u|$ (respectively~$|\nabla v|$) on the closure of~$B$.
Moreover, let
$$
  K_A := \sup_{m_u \leq|\xi|\leq M_u} |A(\xi)|, \qquad K_B := \sup_{m_v \leq|\xi|\leq M_v}|B(\xi)|.
$$
Notice that~$K_A <+\infty$, since~$0\leq m_u \leq M_u <+\infty$;
in fact, if~$\left\lbrace \nabla u=0\right\rbrace =\varnothing$, then $m_u >0$,
whereas, if (A2) holds, then~$A\in L^{\infty}_{loc}(\R^n)$.
In the same way, one has that also~$K_B <+\infty$.

Now, since the assumptions of Lemma~\ref{lemma2} hold, we deduce from~\eqref{system3}
\begin{equation}\label{density}
  \int_{\R^n} A\left(|\nabla u|\right)\nabla u_j \cdot \nabla\phi_k \, dx
+ \int_{\mathcal{D}_{uv}} F_{11}(u,v)\, u_j\,\phi_k + F_{12}(u,v) v_j \phi_k \, dx =0.
\end{equation}
Also,
\begin{eqnarray*}
&& \left|\int_{\R^n} A\left(|\nabla u|\right)\nabla u_j \cdot \left(\nabla\phi_k -\nabla\phi\right) \, dx\right| \\
&&\qquad+ \left|\int_{\mathcal{D}_{uv}} F_{11}(u,v)\, u_j (\phi_k -\phi) + F_{12}(u,v) v_j (\phi_k -\phi) \, dx \right| \\
&\leq& K_A \left(\int_B |\nabla u_j|^2 dx \right)^{1/2}\left(\int_B |\nabla (\phi_k -\phi)|^2 dx \right)^{1/2} \\
&&\qquad+ \left(\int_{B\cap\mathcal{D}_{uv}} |F_{11}(u,v)\, u_j|^2 dx \right)^{1/2} \left(\int_{B\cap\mathcal{D}_{uv}} |\nabla (\phi_k -\phi)|^2 dx \right)^{1/2} \\
 &&\qquad+ \left(\int_{B\cap\mathcal{D}_{uv}} |F_{12}(u,v)\, v_j|^2 dx \right)^{1/2} \left(\int_{B\cap\mathcal{D}_{uv}} |\nabla (\phi_k -\phi)|^2 dx \right)^{1/2},
\end{eqnarray*}
which tends to zero as~$k$ tends to infinity, because of the assumptions on~$u,v$.
The latter consideration and~\eqref{density} give the first equality in~\eqref{system3}.
Reasoning in a similar way, we obtain also the second equality in~\eqref{system3}.
\end{proof}

We will now consider the \emph{tangential gradient} with respect to a regular level set.
Given $w\in C^1 \left(\R^n \right)$, we define the level set of $w$ at $x$ as
\begin{equation}\label{levelset}
L_{w,x} := \left\lbrace y\in\R^n \mathrm{\ s.\ t.\ } w(y)=w(x)\right\rbrace .
\end{equation}
If $\nabla w(x)\neq 0$, $L_{w,x}$ is a hypersurface near $x$ and
one can consider the projection of any vector onto the tangent plane:
in particular, the tangential gradient, which will be denoted as $\nabla_{L_{w,x}}$,
is the projection of the gradient.
This means that, given $f\in C^1 \left( B_r (x)\right)$, for $r>0$,
the tangential gradient is
\begin{equation}\label{lv}
\nabla_{L_{w,x}} f(x) := \nabla f(x) - \left( \nabla f(x) \cdot \frac{\nabla w(x)}{|\nabla w(x)|}\right)
\frac{\nabla w(x)}{|\nabla w(x)|}.
\end{equation}

We will use the following lemma (see Lemma 2.3 in \cite{FSV} for a simple proof):
\begin{lemma}\label{lemma3}
Let $U\subseteq \R^n$ be an open set, $w\in C^2 \left(U\right)$ and $x\in U$ such that $\nabla w(x) \neq 0$.
Then
\begin{eqnarray*}
&& a\left(|\nabla w(x)|\right) \left[ \big| \nabla |\nabla w|(x) \big|^2   - \sum_{j=1}^n |\nabla w_j (x)|^2 \right] \\
&&\qquad - a' \left(|\nabla w(x)|\right) |\nabla w(x)| \, \big| \nabla_{L_{w,x}} |\nabla w|(x) \big|^2   \\
&&=  \left( A\left( \nabla w(x)\right) \left( \nabla |\nabla w|(x)\right) \right)  \cdot
    \left( \nabla |\nabla w|(x) \right)
    - \left( A\left( \nabla w(x)\right) \nabla w_j (x) \right) \cdot \nabla w_j (x),
\end{eqnarray*}
and
\begin{eqnarray*}
&& b\left(|\nabla w(x)|\right) \left[ \big| \nabla |\nabla w|(x) \big|^2   - \sum_{j=1}^n |\nabla w_j (x)|^2 \right]\\
&&\qquad - b' \left(|\nabla w(x)|\right) |\nabla w(x)| \, \big| \nabla_{L_{w,x}} |\nabla w|(x) \big|^2   \\
&&=  \left( B\left( \nabla w(x)\right) \left( \nabla |\nabla w|(x)\right) \right)  \cdot
    \left( \nabla |\nabla w|(x) \right)
    - \left( B\left( \nabla w(x)\right) \nabla w_j (x) \right) \cdot \nabla w_j (x).
\end{eqnarray*}
\end{lemma}

Given $y\in L_{w,x} \cap \left\lbrace \nabla w \neq 0\right\rbrace$,
let $k_{1,w}(y), \ldots, k_{n-1, w}(y)$ denote the principal curvatures of $L_{w,x}$ at $y$.

By using formula $(2.1)$ of \cite{SZ1}, tangential gradients and curvatures
may be conveniently related in the following way:
\begin{equation}\label{curv}
\sum_{j=1}^{n} |\nabla w_j |^2 - \big| \nabla_{L_{w,x}} |\nabla w|\big|^2 - \big| \nabla |\nabla w|\big|^2
= |\nabla w|^2 \sum_{l=1}^{n-1} k_{l,w}^2 ,
\end{equation}
on $\left\lbrace \nabla w \neq 0\right\rbrace$,
for any $w\in C^2 \left(\left\lbrace \nabla w \neq 0\right\rbrace\right)$.

\section{Monotone solutions}\label{sec:mon}
Recalling the definition of monotone solution given in \eqref{monotonicity},
in this section we obtain some geometric inequalities.

\begin{prop}\label{prop1}
Let $\Omega\subseteq\R^n$ be open (not necessarily bounded).
Let~$(u, v)$ be a solution of~\eqref{system}, with $u,v\in C^2 \left(\Omega\right)$,
and $\nabla u, \nabla v \in W^{1,2}_{loc}\left(\Omega\right)$.
Suppose that the monotonicity condition \eqref{monotonicity} holds.

Then,
\begin{equation}\begin{split}\label{estmon}
&\int_{\Omega} \left( A\left( \nabla u(x)\right) \nabla \phi (x)\right) \cdot \nabla\phi (x) \, dx\\
   &\qquad +   \int_{\Omega\cap\mathcal{D}_{uv}} F_{11}(u,v) \phi^2 (x)   +  F_{12}(u,v) \frac{v_n}{u_n} \phi^2 (x) \, dx \geq 0,\\{\mbox{and }}
&\int_{\Omega} \left( B\left( \nabla v(x)\right) \nabla \psi (x)\right) \cdot \nabla\psi (x) \, dx\\
    &\qquad +   \int_{\Omega\cap\mathcal{D}_{uv}} F_{12}(u,v)\frac{u_n}{v_n} \psi^2 (x)   +  F_{22}(u,v) \psi^2 (x) \, dx \geq 0,
\end{split}
\end{equation}
for any locally Lipschitz functions~$\phi, \psi : \Omega\rightarrow\R$
whose supports are compact and contained in $\Omega$.
\end{prop}

\begin{proof}
By Lemma \ref{lemma2bis}, we have that $u_n$ satisfies \eqref{system3}.
We use~$\frac{\phi^2}{u_n}$ as test function in the first equality in \eqref{system3}:
\begin{eqnarray*}
&& \int_{\Omega\cap\mathcal{D}_{uv}} F_{11}(u,v) \phi^2 + F_{12}(u,v) \frac{v_n}{u_n} \phi^2 \, dx \\
&=&
- \int_{\Omega} \left( A\left( \nabla u \right) \nabla u_n\right)  \cdot \nabla\left(\frac{\phi^2}{u_n}\right)\, dx \\
&=& - \int_{\Omega}\left(  A\left( \nabla u \right) \nabla u_n \right) \cdot
           \left( \frac{2\phi\nabla\phi\, u_n - \phi^2 \nabla u_n}{u_n^2}\right)  \, dx \\
&=& \int_{\Omega} \left( A\left( \nabla u \right) \nabla u_n\right) \cdot \nabla u_n  \frac{\phi^2}{u_n^2}
          -2     \left( A\left( \nabla u \right) \nabla u_n\right) \cdot \nabla\phi   \frac{\phi}{u_n} \, dx \\
&&\qquad +   \int_{\Omega}  \left( A\left( \nabla u \right) \nabla\phi\right) \cdot \nabla\phi  -
 \left( A\left( \nabla u \right) \nabla \phi\right) \cdot \nabla\phi \, dx
 \\
&=&   \int_{\Omega}  A\left( \nabla u \right)\left( \nabla u_n \frac{\phi}{u_n} - \nabla\phi\right)  \cdot
          \left( \nabla u_n \frac{\phi}{u_n} - \nabla\phi\right)   - \left(A\left( \nabla u \right) \nabla \phi\right) \cdot \nabla\phi \, dx \\
&\geq&  - \int_{\Omega} \left(A\left( \nabla u \right) \nabla \phi\right) \cdot \nabla\phi \, dx,
\end{eqnarray*}
since \eqref{Axi} holds.
This implies the first inequality in~\eqref{estmon}.

Using~$\frac{\psi^2}{v_n}$ as test function in the second equality in \eqref{system3},
and reasoning as above,
we obtain also the second inequality in~\eqref{estmon}.
\end{proof}

In the subsequents Theorem \ref{T1} and Corollary \ref{cor1} we obtain some inequalities
which involve the principal curvature of the level sets and the tangential gradient of the solution.

\begin{theorem}\label{T1}
Let $\Omega\subseteq\R^n$ be open (not necessarily bounded).
Let~$(u, v)$ be a weak solution of~\eqref{system}, with
$u, v\in C^2 (\Omega)$,
and $\nabla u, \nabla v \in W^{1,2}_{loc}\left(\Omega\right)$.
Suppose that the monotonicity condition \eqref{monotonicity} holds.

For any $x\in\Omega$ let $L_{u,x}$ and $L_{v,x}$ denote the level set of $u$ and $v$ respectively at $x$,
according to \eqref{levelset}.

Let also $\lambda_1 (|\xi |), \lambda_2 (|\xi |), \gamma_1 (|\xi |), \gamma_2 (|\xi |)$
be as in \eqref{lambda} and \eqref{gamma}.

Then,
\begin{equation}\begin{split}\label{u}
&  \int_{\Omega}  \left[ \lambda_1 \left( |\nabla u(x)|\right)  \big| \nabla_{L_{u,x}} |\nabla u| (x)\big|^2
     + \lambda_2 \left( |\nabla u(x)|\right) |\nabla u(x)|^2 \sum_{l=1}^{n-1} k_{l,u}^2 \right]  \phi^2(x) \, dx \\
  &   \leq  \int_{\Omega} |\nabla u(x)|^2 \left( A\left( \nabla u(x)\right) \nabla\phi (x) \right) \cdot \nabla\phi (x)  \,dx
     \\
    &\qquad + \int_{\Omega} F_{12}(u,v) \left( \frac{v_n}{u_n} |\nabla u(x)|^2 - \nabla u (x) \cdot \nabla v(x) \right) \phi^2 (x) \, dx, \end{split}
\end{equation}
and
\begin{equation}\begin{split}\label{v}
 & \int_{\Omega}  \left[ \gamma_1 \left( |\nabla v(x)|\right)  \big| \nabla_{L_{v,x}} |\nabla v| (x)\big|^2
     + \gamma_2 \left( |\nabla v(x)|\right) |\nabla v(x)|^2 \sum_{l=1}^{n-1} k_{l,v}^2 \right]  \psi^2(x) \, dx \\
   &  \leq  \int_{\Omega} |\nabla v(x)|^2 \left( B\left( \nabla v(x)\right) \nabla\psi (x) \right) \cdot \nabla\psi (x)  \, dx
     \\
    &\qquad + \int_{\Omega} F_{12}(u,v) \left( \frac{u_n}{v_n} |\nabla v(x)|^2 - \nabla u (x) \cdot \nabla v(x) \right) \psi^2 (x) \, dx, \end{split}
\end{equation}
for any locally Lipschitz functions $\phi, \psi :\Omega\rightarrow\R$
whose supports are compact and contained in $\Omega$.
\end{theorem}
\begin{proof} We prove first \eqref{u}.
By using the first inequality in \eqref{estmon} with $|\nabla u| \phi$ as test function, we have that
\begin{eqnarray}\label{est1}
0 &\leq&  \int_{\Omega} \left( A\left( \nabla u(x)\right) \nabla\left( |\nabla u(x)| \phi(x)\right) \right)
     \cdot \nabla\left( |\nabla u(x)| \phi(x)\right) \, dx   \nonumber\\
&& +  \int_{\Omega\cap\mathcal{D}_{uv}} F_{11}(u,v) |\nabla u(x)|^2 \phi^2 (x)   +  F_{12}(u,v) \frac{v_n}{u_n} |\nabla u(x)|^2 \phi^2 (x) \, dx \nonumber\\
&=&   \int_{\Omega}  \phi^2 (x) \left( A\left( \nabla u(x)\right) \nabla\left( |\nabla u(x)|\right) \right)\cdot \nabla\left( |\nabla u(x)|\right) \nonumber
\\&& +  |\nabla u(x)|^2  \left( A\left( \nabla u(x)\right) \nabla\phi (x)\right)  \cdot\nabla\phi (x) \nonumber\\
&& + \frac{1}{2} \left( A\left( \nabla u(x)\right) \nabla\left( \phi^2 (x)\right) \right)
\cdot \nabla\left( |\nabla u(x)|^2\right)  \, dx \nonumber\\
&& +   \int_{\Omega\cap\mathcal{D}_{uv}} F_{11}(u,v) |\nabla u(x)|^2 \phi^2 (x)   +  F_{12}(u,v) \frac{v_n}{u_n} |\nabla u(x)|^2 \phi^2 (x)dx.
\end{eqnarray}
Now, since Lemma \ref{lemma2bis} holds, we can use~$u_j \phi^2$ as test function in the first
equality in \eqref{system3}:
\begin{eqnarray*}
    && \int_{\Omega\cap\mathcal{D}_{uv}} F_{11}(u,v) u_j^2 (x) \phi^2 (x) + F_{12}(u,v) u_j (x) v_j (x) \phi^2 (x) \, dx \\
    &=& - \int_{\Omega} \left( A\left( \nabla u(x)\right) \nabla u_j (x)\right)
              \cdot \nabla\left( u_j (x) \phi^2 (x)\right) \, dx\\
    &=&  - \int_{\Omega} \phi^2 (x)  \left( A\left( \nabla u(x)\right) \nabla u_j (x)\right) \cdot \nabla u_j (x) \\
      &&\qquad  +   \frac{1}{2} \left( A\left( \nabla u(x)\right) \nabla \left( \phi^2 (x)\right) \right)
            \cdot \nabla\left( u_j^2 (x) \right) \, dx .
\end{eqnarray*}
We sum over $j$ and use \eqref{est1} to see that
\begin{eqnarray*}
&&  \int_{\Omega} \phi^2 (x) \sum_{j=1}^n \left( A\left( \nabla u(x)\right) \nabla u_j (x)\right) \cdot \nabla u_j (x) \\
   &&\quad  +   \frac{1}{2} \left( A\left( \nabla u(x)\right) \nabla \left( \phi^2 (x)\right) \right)
            \cdot \nabla\left( |\nabla u|^2 (x) \right) \, dx \\
    &  =&  - \int_{\Omega\cap\mathcal{D}_{uv}} F_{11}(u,v) |\nabla u(x)|^2 \phi^2 (x) + F_{12}(u,v) \nabla u(x)\cdot \nabla v(x) \phi^2 (x) \, dx \\
     &\leq&  \int_{\Omega}  \phi^2 (x) \left( A\left( \nabla u(x)\right) \nabla\left( |\nabla u(x)|\right) \right)
     \cdot \nabla\left( |\nabla u(x)|\right) \\
     &&\quad+  |\nabla u(x)|^2  \left( A\left( \nabla u(x)\right) \nabla\phi (x)\right)  \cdot \nabla\phi (x) \\
     &&\quad+ \frac{1}{2} \left( A\left( \nabla u(x)\right) \nabla\left( \phi^2 (x)\right) \right)
     \cdot \nabla\left( |\nabla u(x)|^2\right)  \, dx \\
    &&\quad+ \int_{\Omega\cap\mathcal{D}_{uv}} F_{11}(u,v) |\nabla u(x)|^2 \phi^2 (x)   +  F_{12}(u,v) \frac{v_n}{u_n} |\nabla u(x)|^2 \phi^2 (x) \, dx\\
     &&\quad - \int_{\Omega\cap\mathcal{D}_{uv}} F_{11}(u,v) |\nabla u(x)|^2 \phi^2 (x) + F_{12}(u,v) \nabla u(x)\cdot \nabla v(x) \phi^2 (x) \, dx \\
     \end{eqnarray*}
     \begin{eqnarray*}
    &=&  \int_{\Omega} \phi^2 (x) \left( A\left( \nabla u(x)\right) \nabla\left( |\nabla u(x)|\right) \right)
     \cdot \nabla\left( |\nabla u(x)|\right) \\
     &&\quad+  |\nabla u(x)|^2  \left( A\left( \nabla u(x)\right) \nabla\phi (x)\right)  \cdot \nabla\phi (x) \\
    &&\quad + \frac{1}{2} \left( A\left( \nabla u(x)\right) \nabla\left( \phi^2 (x)\right) \right)
     \cdot \nabla\left( |\nabla u(x)|^2\right)  \, dx \\
    &&\quad +   \int_{\Omega\cap\mathcal{D}_{uv}} F_{12}(u,v)  \left( \frac{v_n}{u_n} |\nabla u(x)|^2 - \nabla u(x)\cdot\nabla v(x)\right) \phi^2 (x)\, dx.
\end{eqnarray*}
By Lemma \ref{lemma3}, we have that
\begin{eqnarray*}
 && \int_{\Omega} \phi^2 (x) a' \left( |\nabla u(x)|\right) |\nabla u(x)| \big| \nabla_{L_{u,x}} |\nabla u(x)|\big|^2  \\
  &&\qquad -   \phi^2 (x)  a\left( |\nabla u(x)|\right) \left[ \big|\nabla |\nabla u(x)|\big|^2
           - \sum_{j=1}^n |\nabla u_j (x)|^2 \right] \, dx \\
    &\leq &  \int_{\Omega}  |\nabla u(x)|^2  \left( A\left( \nabla u(x)\right) \nabla\phi (x)\right)  \cdot \nabla\phi (x) \, dx\\
    &&\qquad +  \int_{\Omega\cap\mathcal{D}_{uv}} F_{12}(u,v)  \left( \frac{v_n}{u_n} |\nabla u(x)|^2 - \nabla u(x)\cdot\nabla v(x)\right) \phi^2 (x)\, dx.
\end{eqnarray*}
That is, by using \eqref{lambda},
\begin{eqnarray*}
&&   \int_{\Omega}  \phi^2 (x) \lambda_1 (|\nabla u(x)|) \big| \nabla_{L_{u,x}} |\nabla u(x)|\big|^2 \\
    &&  + \phi^2 (x) \lambda_2 (|\nabla u(x)|) \left( \sum_{j=1}^n |\nabla u_j (x)|^2 -
      \big| \nabla_{L_{u,x}} |\nabla u(x)|\big|^2 - \big|\nabla |\nabla u(x)|\big|^2 \right) \, dx\\
    &\leq&  \int_{\Omega}  |\nabla u(x)|^2  \left( A\left( \nabla u(x)\right) \nabla\phi (x)\right)  \cdot \nabla\phi (x) \, dx\\
     &&+  \int_{\Omega\cap\mathcal{D}_{uv}} F_{12}(u,v)  \left( \frac{v_n}{u_n} |\nabla u(x)|^2 - \nabla u(x)\cdot\nabla v(x)\right) \phi^2 (x)\, dx.
\end{eqnarray*}
Notice that \eqref{borel} and Theorem 6.19 of \cite{LL} give that
$$
  \nabla u =0= \nabla v \mathrm{\ almost\ everywhere\ on\ } \mathcal{N}_{uv},
$$
and therefore
\begin{eqnarray*}
  &&\int_{\Omega\cap\mathcal{D}_{uv}} F_{12}(u,v)  \left( \frac{v_n}{u_n} |\nabla u(x)|^2 - \nabla u(x)\cdot\nabla v(x)\right) \phi^2 (x)\, dx  \\
&=& \int_{\Omega} F_{12}(u,v)  \left( \frac{v_n}{u_n} |\nabla u(x)|^2 - \nabla u(x)\cdot\nabla v(x)\right) \phi^2 (x)\, dx.
\end{eqnarray*}
This and \eqref{curv} imply the desired result.

Arguing in a similar way we obtain also \eqref{v}.
\end{proof}

\begin{corollary}\label{cor1}
Let $\Omega\subseteq\R^n$ be open (not necessarily bounded).
Let~$(u, v)$ be a weak solution of~\eqref{system}, with
$u, v\in C^2 (\Omega)$,
and $\nabla u, \nabla v \in W^{1,2}_{loc}\left(\Omega\right)$.
Suppose that the monotonicity condition \eqref{monotonicity} holds and that $F_{12}(u,v)\geq 0$.

For any $x\in\Omega$ let $L_{u,x}$ and $L_{v,x}$ denote the level set of $u$ and $v$ respectively at $x$,
according to \eqref{levelset}.

Let also $\lambda_1 (|\xi |), \lambda_2 (|\xi |), \gamma_1 (|\xi |), \gamma_2 (|\xi |)$
be as in \eqref{lambda} and \eqref{gamma}.

Then,
\begin{equation}\begin{split}\nonumber
  & \int_{\Omega}  \left[ \lambda_1 \left( |\nabla u(x)|\right)  \big| \nabla_{L_{u,x}} |\nabla u| (x)\big|^2
     + \lambda_2 \left( |\nabla u(x)|\right) |\nabla u(x)|^2 \sum_{l=1}^{n-1} k_{l,u}^2 \right]  \varphi^2(x) \, dx \\
   &\qquad  +   \int_{\Omega}  \left[ \gamma_1 \left( |\nabla v(x)|\right)  \big| \nabla_{L_{v,x}} |\nabla v| (x)\big|^2
     + \gamma_2 \left( |\nabla v(x)|\right) |\nabla v(x)|^2 \sum_{l=1}^{n-1} k_{l,v}^2 \right]  \varphi^2(x) \, dx \\
    & \leq  \int_{\Omega} |\nabla u(x)|^2 \left( A\left( \nabla u(x)\right) \nabla\varphi (x) \right) \cdot \nabla\varphi (x) \, dx\\
  &\qquad         + \int_{\Omega} |\nabla v(x)|^2 \left( B\left( \nabla v(x)\right) \nabla\varphi (x) \right) \cdot \nabla\varphi (x)  \, dx, \end{split}
\end{equation}
for any locally Lipschitz function $\varphi :\Omega\rightarrow\R$ whose support is
compact and contained in $\Omega$.
\end{corollary}

\begin{proof}
By summing up the inequalities in \eqref{u} and \eqref{v}, we have that, for any $\varphi$ as in the corollary,
\begin{eqnarray*}
  && \int_{\Omega}  \left[ \lambda_1 \left( |\nabla u(x)|\right)  \big| \nabla_{L_{u,x}} |\nabla u| (x)\big|^2
     + \lambda_2 \left( |\nabla u(x)|\right) |\nabla u(x)|^2 \sum_{l=1}^{n-1} k_{l,u}^2 \right]  \varphi^2(x) \, dx \\
    && +   \int_{\Omega}  \left[ \gamma_1 \left( |\nabla
v(x)|\right)  \big| \nabla_{L_{v,x}} |\nabla v| (x)\big|^2
     + \gamma_2 \left( |\nabla v(x)|\right) |\nabla v(x)|^2 \sum_{l=1}^{n-1} k_{l,v}^2 \right]  \varphi^2(x) \, dx \\
     &\leq& \int_{\Omega} |\nabla u(x)|^2 \left( A\left( \nabla u(x)\right) \nabla\varphi (x) \right) \cdot \nabla\varphi (x)  \,dx \\
    && + \int_{\Omega} |\nabla v(x)|^2 \left( B\left( \nabla
v(x)\right) \nabla\varphi (x) \right) \cdot \nabla\varphi (x)  \, dx \\
    && + \int_{\Omega} F_{12}(u,v) \left( \frac{v_n}{u_n} |\nabla
u(x)|^2- 2 \nabla u (x) \cdot \nabla v(x)  +
     \frac{u_n}{v_n} |\nabla v(x)|^2 \right) \varphi^2 (x) \, dx \\
    &=&   \int_{\Omega} |\nabla u(x)|^2 \left( A\left( \nabla u(x)\right) \nabla\varphi (x) \right) \cdot \nabla\varphi (x)  \,dx \\
     && + \int_{\Omega} |\nabla v(x)|^2 \left( B\left( \nabla
v(x)\right) \nabla\varphi (x) \right) \cdot \nabla\varphi (x)  \, dx \\
    && - \int_{\Omega} F_{12}(u,v) \left( \frac{-v_n}{u_n}
|\nabla u(x)|^2 + 2 \nabla u (x) \cdot \nabla v(x)  +
     \frac{u_n}{-v_n} |\nabla v(x)|^2 \right) \varphi^2 (x) \, dx \\
   &=& \int_{\Omega} |\nabla u(x)|^2 \left( A\left( \nabla u(x)\right) \nabla\varphi (x) \right) \cdot \nabla\varphi (x)  \,dx \\
    && + \int_{\Omega} |\nabla v(x)|^2 \left( B\left( \nabla
v(x)\right) \nabla\varphi (x) \right) \cdot \nabla\varphi (x)  \, dx \\
   &&  - \int_{\Omega} F_{12}(u,v) \left|\sqrt{\frac{-v_n}{u_n} }
\nabla u(x) +
   \sqrt{  \frac{u_n}{-v_n}} \nabla v(x) \right|^2  \varphi^2 (x) \, dx ,
\end{eqnarray*}
which gives the conclusion, since $F_{12}(u,v)\geq 0$.
\end{proof}

\section{Stable solutions}\label{sec:stable}
In this section we obtain some geometric inequalities for stable solutions of \eqref{system}.
Since we will use the stability condition~\eqref{stable1}
with a less regular test functions, we need to state the following:
\begin{lemma}\label{lemma4}
Let~$(u,v)$ be a stable weak solution of~\eqref{system}
such that $u,v\in C^1 (\R^n)$.
Suppose that either (A2) holds or that $\left\lbrace \nabla u = 0 \right\rbrace =\varnothing$,
and that either (B2) holds or that $\left\lbrace \nabla v = 0 \right\rbrace =\varnothing$.
Then, the stability condition~\eqref{stable1} holds for any~$\phi, \psi\in W^{1,2}_0 (B)$,
and any ball~$B\subset\R^n$.
\end{lemma}

\begin{proof}
As in the proof of Lemma~\ref{lemma2bis}, we introduce~$m_u, M_u, m_v, M_v, K_A, K_B$, and notice that, under the hypotheses of Lemma \ref{lemma4},~$K_A, K_B <+\infty$.
Moreover, given~$\phi, \psi\in W^{1,2}_0 (B)$, we consider two sequences~$\phi_k, \psi_k \in C^{\infty}_0 (B)$ which converge to~$\phi, \psi$ respectively in~$W^{1,2}_0 (B)$.

Therefore,
\begin{eqnarray*}
 && \left|\int_{\R^n} \left( A\left( \nabla u\right) \nabla \phi_k \right) \cdot \nabla\phi_k \, dx
 - \int_{\R^n} \left( A\left( \nabla u\right) \nabla \phi \right) \cdot \nabla\phi \, dx\right| \\
&\leq & \int_B |A(\nabla u)|\, |\nabla (\phi_k -\phi)|\, |\nabla\phi_k| + |A(\nabla u)|\, |\nabla\phi|\, |\nabla (\phi_k -\phi)| \, dx \\
&\leq & K_A \left(\int_B |\nabla(\phi_k -\phi)|^2 dx\right)^{1/2} \left[ \left(\int_B |\nabla\phi_k|^2 dx\right)^{1/2} + \left(\int_B |\nabla\phi|^2 dx\right)^{1/2}\right],
\end{eqnarray*}
which tends to zero as~$k$ tends to infinity.

Similarly, one obtains
\begin{eqnarray*}
 && \left|\int_{\R^n} \left( B\left( \nabla v\right) \nabla \psi_k \right) \cdot \nabla\psi_k \, dx
 - \int_{\R^n} \left( B\left( \nabla v\right) \nabla \psi \right) \cdot \nabla\psi \, dx \right| \\
&\leq & K_B \left(\int_B |\nabla(\psi_k -\psi)|^2 dx\right)^{1/2} \left[ \left(\int_B |\nabla\psi_k|^2 dx\right)^{1/2} + \left(\int_B |\nabla\psi|^2 dx\right)^{1/2}\right],
\end{eqnarray*}
which again tends to zero.

Moreover, one has that, as~$k$ tends to infinity,
$$
 \int_{\mathcal{D}_{uv}} F_{11}(u,v)\phi_k^2 \, dx \rightarrow \int_{\mathcal{D}_{uv}} F_{11}(u,v)\phi^2 \, dx,
$$
and
$$
  \int_{\mathcal{D}_{uv}} F_{22}(u,v) \psi_k^2 \, dx \rightarrow \int_{\mathcal{D}_{uv}} F_{22}(u,v) \psi^2 \, dx.
$$

Finally,
\begin{eqnarray*}
 && \left|\int_{\mathcal{D}_{uv}} F_{12}(u,v) \phi_k \, \psi_k \, dx -\int_{\mathcal{D}_{uv}} F_{12}(u,v) \phi \, \psi \, dx \right| \\
 &\leq & C \int_{B} |\phi_k \, \psi_k - \phi \, \psi | dx \leq C \int_{B} |\phi_k|\, |\psi_k -\psi| + |\psi| \, |\phi_k -\phi| dx \\
 &\leq & C \left(\int_{B}|\phi_k|^2 dx\right)^{1/2} \left(\int_{B}|\psi_k -\psi|^2 dx\right)^{1/2} \\
 && \qquad + \left(\int_{B}|\psi|^2 dx\right)^{1/2} \left(\int_{B}|\phi_k -\phi|^2 dx\right)^{1/2},
\end{eqnarray*}
which converges to zero as~$k$ tends to infinity.
This concludes the proof.
\end{proof}

We prove next that, under suitable assumptions, a monotone solution of \eqref{system} is also stable.
\begin{prop} \label{prop2}
Let~$(u, v)$ be a weak solution of~\eqref{system}, with
$u, v\in C^2 (\R^n)$,
and $\nabla u, \nabla v \in W^{1,2}_{loc}\left(\R^n\right)$.
Suppose that the monotonicity condition \eqref{monotonicity} holds,
and that $F_{12}(u,v)\geq 0$.
Then~$(u, v)$ is a stable solution.
\end{prop}

\begin{proof}
By summing up the inequalities in \eqref{estmon}, we have
\begin{eqnarray*}
0 &\leq & \int_{\R^n}  \left( A\left( \nabla u(x)\right) \nabla \phi (x)\right) \cdot \nabla\phi (x)
          + \left( B\left( \nabla v(x)\right) \nabla \psi (x)\right) \cdot \nabla\psi (x) \, dx \\
        &&\qquad  +   \int_{\mathcal{D}_{uv}} F_{11}(u,v) \phi^2 (x)   +  F_{22}(u,v) \psi^2 (x)   \\
      &&\qquad \qquad    +  F_{12}(u,v) \left( \frac{v_n}{u_n} \phi^2 (x) + \frac{u_n}{v_n} \psi^2 (x)  \right)   \, dx \\
    &=&   \int_{\R^n}  \left( A\left( \nabla u(x)\right) \nabla \phi (x)\right) \cdot \nabla\phi (x)
          + \left( B\left( \nabla v(x)\right) \nabla \psi (x)\right) \cdot \nabla\psi (x) \, dx \\
        &&\qquad  +   \int_{\mathcal{D}_{uv}} F_{11}(u,v) \phi^2 (x)   +  F_{22}(u,v) \psi^2 (x)  \\
&&\qquad\qquad     -  F_{12}(u,v) \left( \frac{-v_n}{u_n} \phi^2 (x) + \frac{u_n}{-v_n} \psi^2 (x)  \right)   \, dx \\
   &\leq & \int_{\R^n}  \left( A\left( \nabla u(x)\right) \nabla \phi (x)\right) \cdot \nabla\phi (x)
          + \left( B\left( \nabla v(x)\right) \nabla \psi (x)\right) \cdot \nabla\psi (x) \, dx \\
        &&\qquad  +   \int_{\mathcal{D}_{uv}} F_{11}(u,v) \phi^2 (x)   +  F_{22}(u,v) \psi^2 (x)
          + 2 F_{12}(u,v) \phi (x)  \psi (x)    \, dx ,
\end{eqnarray*}
where we have used the monotonicity condition, the fact that $F_{12}(u,v)\geq 0$,
together with
$$   0 \leq \left( \sqrt{\frac{-v_n}{u_n}} \phi (x) + \sqrt{\frac{u_n}{-v_n}} \psi (x)\right)^2 =
    \frac{-v_n}{u_n} \phi^2 (x) + 2 \phi (x) \psi (x) + \frac{u_n}{-v_n} \psi^2 (x) .  $$
This concludes the proof.
\end{proof}

In the subsequents Theorem \ref{T2} and Corollay \ref{cor2}, we prove that a formula obtained in \cite{SZ1, SZ2}
and its extension obtained in \cite{FSV} hold also for a system of the form \eqref{system}.
These formulas relate the stability of the system with the principal curvatures of the corresponding level sets
and with the tangential gradient of the solution.

 \begin{theorem}\label{T2}
Let $\Omega\subseteq\R^n$ be open (not necessarily bounded).
Let~$(u, v)$ be a stable weak solution of~\eqref{system}, with
$u\in C^1(\Omega)\cap C^2 (\Omega \cap \left\lbrace \nabla u \neq 0\right\rbrace)$,
$v\in C^1 \left(\Omega\right)\cap C^2 (\Omega \cap \left\lbrace \nabla v \neq 0\right\rbrace)$,
and $\nabla u, \nabla v \in W^{1,2}_{loc}\left(\Omega\right)$.
Suppose that either (A2) holds or that $\left\lbrace \nabla u = 0 \right\rbrace =\varnothing$,
and that either (B2) holds or that $\left\lbrace \nabla v = 0 \right\rbrace =\varnothing$.

For any $x\in\Omega$ let $L_{u,x}$ and $L_{v,x}$ denote the level set of $u$ and $v$ respectively at $x$,
according to \eqref{levelset}.

Let also $\lambda_1 (|\xi |), \lambda_2 (|\xi |), \gamma_1 (|\xi |), \gamma_2 (|\xi |)$
be as in \eqref{lambda} and \eqref{gamma}.

Then,
\begin{equation}\begin{split}\nonumber
 & \int_{\Omega\cap\left\lbrace \nabla u \neq 0\right\rbrace}  \big[ \lambda_1 \left( |\nabla u(x)|\right)  \big| \nabla_{L_{u,x}} |\nabla u| (x)\big|^2 \\
&\qquad\qquad     + \lambda_2 \left( |\nabla u(x)|\right) |\nabla u(x)|^2 \sum_{l=1}^{n-1} k_{l,u}^2 \big]  \varphi^2(x) \, dx \\
  &\qquad  +      \int_{\Omega\cap\left\lbrace \nabla v \neq 0\right\rbrace}  \big[ \gamma_1 \left( |\nabla v(x)|\right)  \big| \nabla_{L_{v,x}} |\nabla v| (x)\big|^2 \\
 &\qquad\qquad + \gamma_2 \left( |\nabla v(x)|\right) |\nabla v(x)|^2 \sum_{l=1}^{n-1} k_{l,v}^2 \big]  \varphi^2(x) \, dx \\
 &    \leq   \int_{\Omega} |\nabla u(x)|^2 \left( A\left( \nabla u(x)\right) \nabla\varphi (x) \right) \cdot \nabla\varphi (x)\, dx \\
   &\qquad  +  \int_{\Omega} |\nabla v(x)|^2 \left( B\left( \nabla v(x)\right) \nabla\varphi (x) \right) \cdot \nabla\varphi (x)  \, dx \\
    &\qquad + 2 \int_{\Omega} F_{12}(u,v) \left( |\nabla u(x)| \, |\nabla v(x)| - \nabla u (x) \cdot \nabla v(x) \right) \varphi^2 (x) \, dx, \end{split}
\end{equation}
for any locally Lipschitz function~$\varphi:\Omega\rightarrow\R$ whose support is compact and
contained in~$\Omega$.
\end{theorem}

\begin{proof}
Since the maps $x\mapsto u_j (x)$ and $x\mapsto |\nabla u(x)|$ belong to $W^{1,2}_{loc}(\R^n)$,
by using Stampacchia's Theorem (see Theorem 6.19 in \cite{LL}) we have that
$$
  \nabla |\nabla u| = 0 \mathrm{\ almost\ everywhere\ on\ } \left\lbrace |\nabla u| =0\right\rbrace
$$
and, for any $j=1, \ldots , n$,
$$
  \nabla u_j =0 \mathrm{\ almost\ everywhere\ on\ } \left\lbrace |\nabla u| =0\right\rbrace \subseteq \left\lbrace u_j =0\right\rbrace .
$$

Now, we take~$\phi =u_j \varphi^2$ in the first equality in~\eqref{system3} and
we sum over~$j$ to obtain
\begin{equation}\begin{split}\label{st1}
& \sum_j \int_{\R^n} \left(A(\nabla u) \nabla u_j \right)\cdot \nabla (u_j \varphi^2 )\, dx \\
&\qquad + \int_{\mathcal{D}_{uv}} F_{11}(u,v) |\nabla u|^2 \varphi^2 + F_{12}(u,v)\nabla u\cdot\nabla v \, \varphi^2 \, dx =0.
\end{split}\end{equation}
Notice that~\eqref{borel} and Theorem 6.19 of \cite{LL} give that
$$
  \nabla u =0=\nabla v \mathrm{\ almost\ everywhere\ on\ } \mathcal{N}_{uv},
$$
and therefore
\begin{eqnarray*}
 && \int_{\mathcal{D}_{uv}} F_{11}(u,v) |\nabla u|^2 \varphi^2 + F_{12}(u,v)\nabla u\cdot\nabla v \, \varphi^2 \, dx  \\
&&\qquad =\int_{\R^n} F_{11}(u,v) |\nabla u|^2 \varphi^2 + F_{12}(u,v)\nabla u\cdot\nabla v \, \varphi^2 \, dx.
\end{eqnarray*}

Taking~$\psi =v_j \varphi^2$ in the second equality in~\eqref{system3} and summing over~$j$,
we obtain
\begin{equation}\begin{split}\label{st2}
 &\sum_j \int_{\R^n} \left(B(\nabla v) \nabla v_j \right)\cdot \nabla (v_j \varphi^2 )\, dx \\&\qquad +
\int_{\R^n} F_{12}(u,v)\nabla u\cdot\nabla v\, \varphi^2 + F_{22}(u,v) |\nabla v|^2 \varphi^2 \, dx =0.
\end{split}\end{equation}

Now, we exploit the stability condition~\eqref{stable1} with~$\phi=|\nabla u|\varphi$
and~$\psi=|\nabla v|\varphi$.
Note that this choice is possible, thanks to Lemma~\ref{lemma4}, and gives
\begin{eqnarray}\label{st3}
 0&\leq & \int_{\R^n} \left(A(\nabla u)\nabla(|\nabla u|\varphi)\right)\cdot\nabla(|\nabla u|\varphi)
+ \left(B(\nabla v)\nabla(|\nabla v|\varphi)\right)\cdot\nabla(|\nabla v|\varphi) \nonumber\\
&& + F_{11}(u,v)|\nabla u|^2 \varphi^2 + F_{22}(u,v) |\nabla v|^2 \varphi^2
+ 2 F_{12}(u,v) |\nabla u| \, |\nabla v| \varphi^2 \, dx \nonumber\\
&=& \int_{\R^n} |\nabla u|^2 \left(A(\nabla u)\nabla\varphi\right)\cdot\nabla\varphi +
|\nabla v|^2 \left(B(\nabla v)\nabla\varphi\right)\cdot\nabla\varphi \nonumber\\
&& + \varphi^2 \left(A(\nabla u)\nabla |\nabla u|\right)\cdot\nabla |\nabla u| +
\varphi^2 \left(B(\nabla v)\nabla |\nabla v|\right)\cdot\nabla |\nabla v| \nonumber\\
&& + 2\varphi |\nabla u| \left(A(\nabla u)\nabla\varphi\right)\cdot\nabla |\nabla u| +
 2\varphi |\nabla v| \left(B(\nabla v)\nabla\varphi\right)\cdot\nabla |\nabla v| \nonumber\\
&& + F_{11}(u,v) |\nabla u|^2 \varphi^2 + F_{22}(u,v) |\nabla v|^2 \varphi^2
+ 2F_{12}(u,v) |\nabla u|\, |\nabla v| \, \varphi^2 dx.
\end{eqnarray}
By using~\eqref{st1} and~\eqref{st2} in~\eqref{st3}, we get
\begin{eqnarray*}
 0&\leq & \int_{\R^n} |\nabla u|^2 \left(A(\nabla u)\nabla\varphi\right)\cdot\nabla\varphi +
|\nabla v|^2 \left(B(\nabla v)\nabla\varphi\right)\cdot\nabla\varphi \, dx\\
&&\qquad + \int_{\left\lbrace \nabla u\neq 0\right\rbrace }\varphi^2 \left(A(\nabla u)\nabla |\nabla u|\right)\cdot\nabla |\nabla u|\, dx \\ &&\qquad +
\int_{\left\lbrace \nabla v\neq 0\right\rbrace }\varphi^2 \left(B(\nabla v)\nabla |\nabla v|\right)\cdot\nabla |\nabla v|\, dx \\
&&\qquad + \int_{\left\lbrace \nabla u\neq 0\right\rbrace }2\varphi |\nabla u| \left(A(\nabla u)\nabla\varphi\right)\cdot\nabla |\nabla u| - \sum_j \left(A(\nabla u) \nabla u_j \right)\cdot \nabla (u_j \varphi^2 )\, dx \\
&&\qquad + \int_{\left\lbrace \nabla v\neq 0\right\rbrace } 2\varphi |\nabla v| \left(B(\nabla v)\nabla\varphi\right)\cdot\nabla |\nabla v| - \sum_j \left(B(\nabla v) \nabla v_j \right)\cdot \nabla (v_j \varphi^2 )\, dx
\\
&&\qquad + \int_{\R^n} 2F_{12}(u,v) \left(|\nabla u|\, |\nabla v|-\nabla u\cdot\nabla v\right) \, \varphi^2 \, dx \\
 &=& \int_{\R^n} |\nabla u|^2 \left(A(\nabla u)\nabla\varphi\right)\cdot\nabla\varphi +
|\nabla v|^2 \left(B(\nabla v)\nabla\varphi\right)\cdot\nabla\varphi \, dx\\
&&\qquad + \int_{\left\lbrace \nabla u\neq 0\right\rbrace }\varphi^2 \left[\left(A(\nabla u)\nabla |\nabla u|\right)\cdot\nabla |\nabla u| - \sum_j \left(A(\nabla u) \nabla u_j \right)\cdot \nabla u_j \right] \, dx \\
&&\qquad + \int_{\left\lbrace \nabla v\neq 0\right\rbrace }\varphi^2 \left[\left(B(\nabla v)\nabla |\nabla v|\right)\cdot\nabla |\nabla v| -\sum_j \left(B(\nabla v) \nabla v_j \right)\cdot \nabla v_j \right] \, dx \\
&&\qquad + \int_{\R^n} 2F_{12}(u,v) \left(|\nabla u|\, |\nabla v|-\nabla u\cdot\nabla v\right) \, \varphi^2 \, dx.
\end{eqnarray*}
Now, the use of Lemma \ref{lemma3} implies
\begin{eqnarray*}
 0&\leq&  \int_{\R^n} |\nabla u|^2 \left(A(\nabla u)\nabla\varphi\right)\cdot\nabla\varphi +
|\nabla v|^2 \left(B(\nabla v)\nabla\varphi\right)\cdot\nabla\varphi \, dx \\
&&\qquad + \int_{\left\lbrace \nabla u\neq 0\right\rbrace} \varphi^2 \big[ a(|\nabla u|) \left( |\nabla|\nabla u||^2 - \sum_j |\nabla u_j|^2 \right) \\ &&\qquad \qquad - a'(|\nabla u|) |\nabla u|\, |\nabla_{L_{u,x}}|\nabla u||^2 \big] \, dx \\
\end{eqnarray*}
\begin{eqnarray*}
&&\qquad + \int_{\left\lbrace \nabla v\neq 0\right\rbrace} \varphi^2 \big[ b(|\nabla v|) \left( |\nabla|\nabla v||^2 - \sum_j |\nabla v_j|^2 \right) \\ &&\qquad\qquad - b'(|\nabla v|) |\nabla v|\, |\nabla_{L_{v,x}}|\nabla v||^2 \big] \, dx \\
&&\qquad +  \int_{\R^n} 2F_{12}(u,v) \left(|\nabla u|\, |\nabla v|-\nabla u\cdot\nabla v\right) \, \varphi^2 \, dx.
\end{eqnarray*}
That is, using~\eqref{lambda} and~\eqref{gamma}
\begin{eqnarray*}
&& \int_{\R^n} |\nabla u|^2 \left(A(\nabla u)\nabla\varphi\right)\cdot\nabla\varphi +
|\nabla v|^2 \left(B(\nabla v)\nabla\varphi\right)\cdot\nabla\varphi \, dx \\
&&\qquad + \int_{\R^n} 2F_{12}(u,v) \left(|\nabla u|\, |\nabla v|-\nabla u\cdot\nabla v\right) \, \varphi^2 \, dx \\
&\geq& \int_{\left\lbrace \nabla u\neq 0\right\rbrace} \varphi^2 \big[ \lambda_1 (|\nabla u|) |\nabla_{L_{u,x}}|\nabla u||^2 \\&&\qquad \qquad + \lambda_2 (|\nabla u|) \left( \sum_j |\nabla u_j|^2 - |\nabla_{L_{u,x}}|\nabla u||^2 - |\nabla|\nabla u||^2 \right)\big] \, dx \\
&& + \int_{\left\lbrace \nabla v\neq 0\right\rbrace} \varphi^2 \big[ \gamma_1 (|\nabla v|) |\nabla_{L_{v,x}}|\nabla v||^2 \\&&\qquad \qquad + \gamma_2 (|\nabla v|) \left( \sum_j |\nabla v_j|^2 - |\nabla_{L_{v,x}}|\nabla v||^2 - |\nabla|\nabla v||^2 \right)\big] \, dx.
\end{eqnarray*}
This and~\eqref{curv} imply the desired result.
\end{proof}

An immediate consequence of Theorem~\ref{T2} is the following:
\begin{corollary}\label{cor2}
Let $\Omega\subseteq\R^n$ be open (not necessarily bounded).
Let~$(u, v)$ be a stable weak solution of~\eqref{system}, with
$u\in C^1(\Omega)\cap C^2 (\Omega \cap \left\lbrace \nabla u \neq 0\right\rbrace)$,
$v\in C^1 \left(\Omega\right)\cap C^2 (\Omega \cap \left\lbrace \nabla v \neq 0\right\rbrace)$,
and $\nabla u, \nabla v \in W^{1,2}_{loc}\left(\Omega\right)$.
Suppose that either (A2) holds or that $\left\lbrace \nabla u = 0 \right\rbrace =\varnothing$,
and that either (B2) holds or that $\left\lbrace \nabla v = 0 \right\rbrace =\varnothing$.
Moreover, assume that $F_{12}(u,v)\leq 0$.

For any $x\in\Omega$ let $L_{u,x}$ and $L_{v,x}$ denote the level set of $u$ and $v$ respectively at $x$,
according to \eqref{levelset}.

Let also $\lambda_1 (|\xi |), \lambda_2 (|\xi |), \gamma_1 (|\xi |), \gamma_2 (|\xi |)$
be as in \eqref{lambda} and \eqref{gamma}.

Then,

\begin{align*}
 & \int_{\Omega\cap\left\lbrace \nabla u \neq 0\right\rbrace}  \big[ \lambda_1 \left( |\nabla u(x)|\right)  \big| \nabla_{L_{u,x}} |\nabla u| (x)\big|^2 \\
    &\qquad + \lambda_2 \left( |\nabla u(x)|\right) |\nabla u(x)|^2 \sum_{l=1}^{n-1} k_{l,u}^2 \big]  \varphi^2(x) \, dx \\
   & +      \int_{\Omega\cap\left\lbrace \nabla v \neq 0\right\rbrace}  \big[ \gamma_1 \left( |\nabla v(x)|\right)  \big| \nabla_{L_{v,x}} |\nabla v| (x)\big|^2 \\
    &\qquad + \gamma_2 \left( |\nabla v(x)|\right) |\nabla v(x)|^2 \sum_{l=1}^{n-1} k_{l,v}^2 \big]  \varphi^2(x) \, dx \\
    \end{align*}
    \begin{align*}
     &\leq  \int_{\Omega} |\nabla u(x)|^2 \left( A\left( \nabla u(x)\right) \nabla\varphi (x) \right) \cdot \nabla\varphi (x) \, dx \\
     &\qquad +  \int_{\Omega} |\nabla v(x)|^2 \left( B\left( \nabla v(x)\right) \nabla\varphi (x) \right) \cdot \nabla\varphi (x)\, dx,
\end{align*}
for any locally Lipschitz function~$\varphi :\Omega\rightarrow\R$ whose support is compact
and contained in~$\Omega$.
\end{corollary}

\section{Level set analysis}\label{sec:levset}
We recall here the geometric analysis performed in Subsection 2.4 in \cite{FSV}.
In order to make this paper self-contained,
we include the proofs in full detail.

We consider connected components of the level sets
(in the inherited topology).

\begin{lemma}\label{lemma5}
Let~$w\in C^1(\R^n)\cap C^2(\left\lbrace \nabla w\neq 0\right\rbrace)$.
Fix~$\overline{x}\in\R^n$, and suppose that
for any~$x\in L_{w,\overline{x}}\cap\left\lbrace \nabla w\neq 0\right\rbrace$,
we have that~$\nabla_{L_{w,x}}|\nabla w(x)|=0$.

Then,~$|\nabla w|$ is constant on every connected component
of~$L_{w,\overline{x}}\cap\left\lbrace \nabla w\neq 0\right\rbrace$.
\end{lemma}

\begin{proof}
 Since any connected components of~$L_{w,\overline{x}}\cap\left\lbrace \nabla w\neq 0\right\rbrace$ is a regular hypersurface,
any two points in it may be joined by a~$C^1$ path.

We notice that, if~$t_1 >t_0 \in\R$ and~$\sigma\in C^1([t_0,t_1], L_{w,\overline{x}}\cap\left\lbrace \nabla w\neq 0\right\rbrace )$, then
$$
  \frac{d}{dt}|\nabla w(\sigma(t))|=\nabla|\nabla w(\sigma(t))|\cdot\dot{\sigma}(t)=
  \nabla_{L_{w,\overline{x}}}|\nabla w(\sigma(t))|\cdot\dot{\sigma(t)},
$$
thanks to~\eqref{lv}.
As a consequence, if $\sigma\in C^1([t_0,t_1], L_{w,\overline{x}}\cap\left\lbrace \nabla w\neq 0\right\rbrace )$,
then $|\nabla w(\sigma(t))|$ is constant for $t\in[t_0,t_1]$.

Now, we take~$a$ and~$b$ in~$L_{w,\overline{x}}\cap\left\lbrace \nabla w\neq 0\right\rbrace$
and~$\sigma\in C^1([0,1], L_{w,\overline{x}})$ such that~$\sigma(0)=a$
and~$\sigma(1)=b$. Then~$|\nabla w(a)|=|\nabla w(b)|$.
\end{proof}

\begin{corollary}\label{cor3}
Under the assumptions of Lemma~\ref{lemma5},
every connected component of~$L_{w,x}\cap\left\lbrace \nabla w\neq 0\right\rbrace$ is closed in~$\R^n$.
\end{corollary}

\begin{proof}
Let~$M$ be any connected component of~$L_{w,x}\cap\left\lbrace \nabla w\neq 0\right\rbrace$.
With no loss of generality, we suppose that~$M\neq\varnothing$ and take~$z\in M$.

Let~$y\in\partial M$. Then there is a sequence~$x_n \in M$ approaching~$y$, thus
\begin{equation}\label{M}
w(y)=\lim_{n\rightarrow +\infty} w(x_n) = w(z).
\end{equation}
Then, by Lemma~\ref{lemma5}, we have that~$|\nabla w(x_n)|=|\nabla w(z)|$.
So, since~$z\in M$,
\begin{equation}\label{M1}
|\nabla w(y)|=\lim_{n\rightarrow +\infty} |\nabla w(x_n)| = |\nabla w(z)|\neq 0.
\end{equation}
By~\eqref{M} and~\eqref{M1}, we have that~$y\in M$.
\end{proof}

\begin{corollary}\label{cor4}
Let the assumptions of Lemma~\ref{lemma5} hold.
Let~$M$ be a connected component of~$L_{w,x}\cap\left\lbrace \nabla w\neq 0\right\rbrace$.
Suppose that~$M\neq\varnothing$ and~$M$ is contained in a hyperplane~$\pi$.
Then,~$M=\pi$.
\end{corollary}

\begin{proof}
We show that
\begin{equation}\label{M2}
 M \mathrm{\ is\ open\ in\ the\ topology\ of\ } \pi.
\end{equation}
For this, let~$z\in M$.
Then, there exists an open set~$U_1$ of~$\R^n$ such
that~$z\in U_1 \subset\left\lbrace \nabla w\neq 0\right\rbrace$.
Also, by the Implicit Function Theorem, there exists an open set~$U_2$ in~$\R^n$
for which~$z\in U_2$ and~$L_{w,x}\cap U_2$
is a hypersurface.
Since~$M\subseteq\pi$, we have that~$L_{w,x}\cap U_2 \subseteq\pi$,
hence~$L_{w,x}\cap U_2$ is open in the topology of~$\pi$.

Then,~$z\in L_{w,x}\cap U_1 \cap U_2$, which is an open set in~$\pi$.

This proves~\eqref{M2}.

Also,~$M$ is closed in~$\R^n$ and so~$M=M\cap\pi$ is closed in~$\pi$.

Hence,~$M$ is open and closed in~$\pi$.
\end{proof}

\begin{lemma}\label{lemma6}
Let~$w\in C^1(\R^n)\cap C^2(\left\lbrace \nabla w\neq 0\right\rbrace )$
be such that~$\nabla_{L_{w,x}}|\nabla w(x)|=0$ for every~$x\in\left\lbrace \nabla w\neq 0\right\rbrace$, and let~$\overline{x}\in\R^n$.

Suppose that a non-empty connected component~$\overline{L}$
of~$L_{w,\overline{x}}\cap\left\lbrace \nabla w\neq 0\right\rbrace$
has zero principal curvatures at all points.

Then,~$\overline{L}$ is a flat hyperplane.
\end{lemma}

\begin{proof}
We use a standard  differential geometry argument (see, for instance, page 311 in \cite{Se}). Since the principal curvatures vanish identically, the normal of~$\overline{L}$ is constant,
thence~$\overline{L}$ is contained in a hyperplane.

Then, the claim follows from Corollary~\ref{cor4}.
\end{proof}

\begin{lemma}\label{lemma7}
Let~$w\in C^1(\R^n)\cap C^2(\left\lbrace \nabla w\neq 0\right\rbrace )$.
Suppose that
\begin{equation}\begin{split}\label{M3}
\mathrm{any\ connected\ component\ of\ } L_{w,x}\cap\left\lbrace \nabla w\neq 0\right\rbrace \\ \mathrm{\ has\ zero\ principal\ curvatures\ at\ all\ points}
                \end{split}
\end{equation}
and that, for any~$x\in \left\lbrace \nabla w\neq 0\right\rbrace$,
\begin{equation}\label{M4}
 \nabla_{L_{w,x}}|\nabla w(x)|=0.
\end{equation}

Then,~$w$ possesses one-dimensional symmetry, in the sense that there
exists~$\overline{w}:\R\rightarrow\R$ and~$\omega\in S^{n-1}$ in such a way
that~$w(x)=\overline{w}(\omega\cdot x)$, for any~$x\in\R^n$.
\end{lemma}

\begin{proof}
If~$\nabla w(x)=0$ for any~$x\in\R^n$, the one-dimensional symmetry is trivial.

If~$\nabla w(x)\neq 0$, then the connected component
of~$L_{w,\overline{x}}\cap\left\lbrace \nabla w\neq 0\right\rbrace$
passing through~$\overline{x}$ is a hyperplane, thanks to Lemma~\ref{lemma6}.

We notice that all these hyperplanes are parallel,
since connected components cannot intersect.
Moreover,~$w$ is constant on these hyperplanes,
because each of them lies on a level set.

On the other hand,~$w$ is also constant on any other possible hyperplane parallel to
the ones of the above family, since the gradient vanishes identically there.

From this, the one-dimensional symmetry of~$w$ follows by noticing that~$w$
only depends on the orthogonal direction with respect to the above family of
hyperplanes.
\end{proof}

\section{Proof of Theorem \ref{T1D}}\label{sec:proof}
Since either~\eqref{mon} or~\eqref{stab} holds, from Corollaries~\ref{cor1} and~\ref{cor2}, we have
\begin{eqnarray}\label{70}
&& \int_{\left\lbrace \nabla u\neq 0\right\rbrace} \big[\lambda_1(|\nabla u(x)|) |\nabla_{L_{u,x}}|\nabla u|(x)|^2 \nonumber\\
&& \qquad\qquad +\lambda_2(|\nabla u(x)|)|\nabla u(x)|^2 \sum_{l=1}^{n-1} k_{l,u}^2\big]\varphi^2(x) \, dx \nonumber\\
   &&\qquad  +      \int_{\left\lbrace \nabla v \neq 0\right\rbrace}  \big[ \gamma_1 \left( |\nabla v(x)|\right)  \big| \nabla_{L_{v,x}} |\nabla v| (x)\big|^2 \nonumber\\
 &&\qquad\qquad    + \gamma_2 \left( |\nabla v(x)|\right) |\nabla v(x)|^2 \sum_{l=1}^{n-1} k_{l,v}^2 \big]  \varphi^2(x) \, dx \nonumber\\
     &\leq&  \int_{\R^n} |\nabla u(x)|^2 \left( A\left( \nabla u(x)\right) \nabla\varphi (x) \right) \cdot \nabla\varphi (x) \, dx \nonumber\\
    &&\qquad  +  \int_{\R^n} |\nabla v(x)|^2 \left( B\left( \nabla v(x)\right) \nabla\varphi (x) \right) \cdot \nabla\varphi (x)\, dx \nonumber\\
     &\leq& \int_{\R^n} \left( |A(\nabla u(x))|\, |\nabla u(x)|^2 + |B(\nabla v(x))|\, |\nabla v(x)|^2 \right) |\nabla\varphi(x)|^2 \, dx.
\end{eqnarray}

Now, we chose conveniently~$\varphi$ in~\eqref{70}.
For any~$R>1$, we define the function~$\varphi_R$ as
\begin{equation}\label{fiR}
\varphi_R(x):=\left\{
\begin{matrix}
1 & {\mbox{ if $x\in B_{\sqrt R}$,}}\\
2\, \frac{\log R-\log|x|}{\log R}
& {\mbox{ if $x\in B_R\setminus B_{\sqrt R}$,}}\\
0 & {\mbox{ if $x\in \R^n\setminus B_{R}$.}}
\end{matrix}
\right.
\end{equation}
We denote by
$$ \chi_R:=\chi_{B_R\setminus B_{\sqrt R}}. $$
Notice that
$$  |\nabla\varphi_R(x)|=\frac{\chi_R(x)}{2|x|\log R}. $$
Therefore, by using~$\varphi_R$ in~\eqref{70}, we have
\begin{eqnarray}\label{71}
&& \int_{B_{\sqrt R}\cap\left\lbrace \nabla u\neq 0\right\rbrace} \big[\lambda_1(|\nabla u(x)|) |\nabla_{L_{u,x}}|\nabla u|(x)|^2 \nonumber\\
&&\qquad\qquad +\lambda_2(|\nabla u(x)|)|\nabla u(x)|^2 \sum_{l=1}^{n-1} k_{l,u}^2\big] \, dx \nonumber\\
 &&\qquad    +      \int_{B_{\sqrt R}\cap\left\lbrace \nabla v \neq 0\right\rbrace}  \big[ \gamma_1 \left( |\nabla v(x)|\right)  \big| \nabla_{L_{v,x}} |\nabla v| (x)\big|^2 \nonumber\\
&&\qquad\qquad + \gamma_2 \left( |\nabla v(x)|\right) |\nabla v(x)|^2 \sum_{l=1}^{n-1} k_{l,v}^2 \big]   \, dx \nonumber\\
&\leq&  \frac{C}{\log^2 R}\int_{B_R \setminus B_{\sqrt{R}}}  \frac{|A(\nabla u(x))|\, |\nabla u(x) |^2 +|B(\nabla v(x))|\, |\nabla v(x) |^2}{|x|^2}\, dx.
\end{eqnarray}
Letting~$R\rightarrow +\infty$ in~\eqref{71}, by the hypothesis~\eqref{EE}, we obtain
\begin{eqnarray*}
&& \int_{\left\lbrace \nabla u\neq 0\right\rbrace} \big[\lambda_1(|\nabla u(x)|) |\nabla_{L_{u,x}}|\nabla u|(x)|^2 \\ &&\qquad\qquad +\lambda_2(|\nabla u(x)|)|\nabla u(x)|^2 \sum_{l=1}^{n-1} k_{l,u}^2\big] \, dx \\
  &&\qquad  +      \int_{\left\lbrace \nabla v \neq 0\right\rbrace}  \big[ \gamma_1 \left( |\nabla v(x)|\right)  \big| \nabla_{L_{v,x}} |\nabla v| (x)\big|^2 \\
&&\qquad\qquad + \gamma_2 \left( |\nabla v(x)|\right) |\nabla v(x)|^2 \sum_{l=1}^{n-1} k_{l,v}^2 \big]   \, dx =0,
\end{eqnarray*}
which implies that, for any $x\in \left\lbrace \nabla u\neq 0\right\rbrace$,
$$
 \lambda_1(|\nabla u(x)|) |\nabla_{L_{u,x}}|\nabla u|(x)|^2 +\lambda_2(|\nabla u(x)|)|\nabla u(x)|^2 \sum_{l=1}^{n-1} k_{l,u}^2 =0,
$$
and, for any $x\in \left\lbrace \nabla v \neq 0\right\rbrace$,
$$
  \gamma_1 \left( |\nabla v(x)|\right)  \big| \nabla_{L_{v,x}} |\nabla v| (x)\big|^2
     + \gamma_2 \left( |\nabla v(x)|\right) |\nabla v(x)|^2 \sum_{l=1}^{n-1} k_{l,v}^2 =0.
$$

Recalling the definition of $\lambda_1, \lambda_2, \gamma_1, \gamma_2$ given in \eqref{lambda} and \eqref{gamma},
and the assumptions \eqref{ab1} and \eqref{ab2}, the last two equalities imply that
$$
  \nabla_{L_{u,x}}|\nabla u|(x)=0, \qquad k_{1,u}=\ldots=k_{n-1,u}=0,
$$
for any $x\in \left\lbrace \nabla u\neq 0\right\rbrace$, and that
$$
  \nabla_{L_{v,x}}|\nabla v|(x)=0, \qquad k_{1,v}=\ldots=k_{n-1,v}=0,
$$
for any $x\in \left\lbrace \nabla v\neq 0\right\rbrace$.
This means that $u,v$ satisfy \eqref{M3} and \eqref{M4}.
Hence, by Lemma \ref{lemma7} we obtain that there exist~$\overline{u}, \overline{v}:\R\rightarrow\R$
and~$\omega_{u}, \omega_{v}\in S^{n-1}$ in such a way that~
$(u(x), v(x))=(\overline{u}(\omega_{u}\cdot x), \overline{v}(\omega_{v}\cdot x))$, for any~$x\in\R^n$,
which proves the first part of Theorem \ref{T1D}.
\medskip

Now, we assume that condition \eqref{monF12} holds.
Since~$(u,v)$ has a one dimensional symmetry, by summing up \eqref{u} and \eqref{v} we obtain that
\begin{equation}\begin{split}\label{u111}
& \int_{\R^n} F_{12}(u,v) \left| \sqrt{\frac{-v_n}{u_n}} \nabla u(x) + \sqrt{\frac{u_n}{-v_n}} \nabla v(x) \right|^2 \varphi^2 (x) \, dx,
\\  &\qquad   \leq  \int_{\R^n} |\nabla u(x)|^2 \left( A\left( \nabla u(x)\right) \nabla\varphi (x) \right) \cdot \nabla\varphi (x)  \,dx
  \\ & \qquad\qquad + \int_{\R^n} |\nabla v(x)|^2 \left( B\left( \nabla v(x)\right) \nabla\varphi (x) \right) \cdot \nabla\varphi (x)  \, dx.
\end{split}
\end{equation}
Choosing the test function $\varphi$ as in \eqref{fiR} and reasoning as above, we obtain
$$ \int_{\R^n} F_{12}(u,v) \left| \sqrt{\frac{-v_n}{u_n}} \nabla u(x) + \sqrt{\frac{u_n}{-v_n}} \nabla v(x) \right|^2  \, dx =0, $$
which implies that
$$  F_{12}(u,v) \left| \sqrt{\frac{-v_n}{u_n}} \nabla u(x) + \sqrt{\frac{u_n}{-v_n}} \nabla v(x) \right|^2 =0 \qquad \mbox{\ a.e.} $$
Since \eqref{monF12} holds, there exists $x_0\in\Omega'$ such that
$F_{12}(u(x_0),v(x_0))>0$. Therefore,
$$  \sqrt{\frac{-v_n(x_0)}{u_n(x_0)}} \nabla u(x_0) + \sqrt{\frac{u_n(x_0)}{-v_n(x_0)}} \nabla v(x_0) =0, $$
which gives that $\nabla u(x_0)=h(x_0)\nabla v(x_0)$, for some function $h$.
Since we know that~$(u,v)$ has a one dimensional symmetry, this implies that $\omega_{u}=\omega_{v}$.

Finally, we assume that condition \eqref{stabF12} holds. 
Arguing as in the proof of Theorem~$1.8$ in~\cite{DP} 
(see the comments after formula~$(8.5)$), one can prove that 
\begin{equation}\begin{split}\label{qqq}
&\mbox{there\ exists\ a\ non-empty\ open\ set\ $\Omega''\subset\R^2$\ such\ that}\\
&\mbox{$u(x)\in I_u$,\ $v(x)\in I_v$,\ $\nabla u(x)\neq 0$\ and\ $\nabla v(x)\neq 0$\ for\ all\ $x\in\Omega''$}.
\end{split}\end{equation}
Now, reasoning as above, from Theorem \ref{T2} we obtain that
\begin{equation}\begin{split}\nonumber
 &  -2 \int_{\R^n} F_{12}(u,v) \left( |\nabla u(x)| \, |\nabla v(x)| - \nabla u (x) \cdot \nabla v(x) \right) \varphi^2 (x) \, dx\\
 &  \qquad  \leq   \int_{\R^n} |\nabla u(x)|^2 \left( A\left( \nabla u(x)\right) \nabla\varphi (x) \right) \cdot \nabla\varphi (x)\, dx \\
   &\qquad\qquad  +  \int_{\R^n} |\nabla v(x)|^2 \left( B\left( \nabla v(x)\right) \nabla\varphi (x) \right) \cdot \nabla\varphi (x)  \, dx.
  \end{split}
\end{equation}
We choose the test function $\varphi$ as in \eqref{fiR} and we use that fact that $F_{12}(u,v)\leq 0$ to get
$$  F_{12}(u,v) \left( |\nabla u(x)| \, |\nabla v(x)| - \nabla u (x) \cdot \nabla v(x) \right) =0 \qquad \mbox{\ a.e.} $$
By~\eqref{qqq} and~\eqref{stabF12}, there exists $x_1\in\Omega''$ such that
$F_{12}(u(x_1),v(x_1))<0$. Hence
\begin{equation}\begin{split}\nonumber
& |\nabla u(x_1)| \, |\nabla v(x_1)| - \nabla u (x_1) \cdot \nabla v(x_1) \\ &\qquad  =
|\nabla u(x_1)| \, |\nabla v(x_1)| -|\nabla u(x_1)| \, |\nabla v(x_1)| \, \frac{\nabla u(x_1)}{|\nabla u(x_1)|}\cdot \frac{\nabla v(x_1)}{|\nabla v(x_1)|}=0,
\end{split}\end{equation}
which implies that
$$ \frac{\nabla u(x_1)}{|\nabla u(x_1)|}\cdot \frac{\nabla v(x_1)}{|\nabla v(x_1)|}=1. $$
Since we know that~$(u,v)$ has a one dimensional symmetry, this implies that $\omega_{u}=\omega_{v}$.
This concludes the proof of Theorem \ref{T1D}.

\section{An application}\label{sec:appl}
In this section, we use the result stated in Theorem \ref{T1D} to obtain a proof of a conjecture of De Giorgi
for the system \eqref{system} in $\R^2$.

\begin{theorem} \label{Tn2}
Let~$n=2$, and let~$(u, v)$ be a weak solution of~\eqref{system},
with $u\in C^{1}(\R^2)\cap C^2(\left\lbrace \nabla u\neq 0\right\rbrace)$,
$v\in C^1(\R^2)\cap C^2(\left\lbrace \nabla v\neq 0\right\rbrace)$, and
$\nabla u, \nabla v\in L^{\infty}(\R^2)\cap W^{1,2}_{loc}(\R^2)$.

Suppose that either (A1) or (A2) holds, and that either (B1) or (B2) holds.

Assume that either
\begin{equation}\nonumber
\mbox{the monotonicity condition~\eqref{monotonicity} holds, and $F_{12}(u,v)\geq 0$\ in\ $\Im(u,v)$}, \end{equation}
or
\begin{equation}\nonumber
   \mbox{(u,v) is stable, and $F_{12}(u,v)\leq 0$\ in\ $\Im(u,v)$}. \end{equation}
Then~$(u, v)$ has one-dimensional symmetry, in the sense that there exist~$\overline{u}, \overline{v}:\R\rightarrow\R$ and~$\omega_{u}, \omega_{v}\in S^{n-1}$ in such a way that~
$(u(x), v(x))=(\overline{u}(\omega_{u}\cdot x), \overline{v}(\omega_{v}\cdot x))$, for any~$x\in\R^n$.

Moreover, if we assume in addition that either
\begin{equation}\begin{split}\nonumber
&\mbox{the\ monotonicity\ condition\ \eqref{monotonicity}\ holds,\ and\ there\ exists\ a\  non-empty}\\
&\mbox{open\ set\ $\Omega'\subseteq\R^n$\ such\ that\ $F_{12}(u(x),v(x))>0$\ for\ any\ $x\in\Omega'$}  , 
\end{split}\end{equation}
or
\begin{equation}\begin{split}\nonumber
&\mbox{$(u,v)$\ is\ stable,\ and\ there\ exist\ two\ open intervals\ $I_u,I_v\subset\R$}\\
&\mbox{such\ that\ $\left(I_u\times I_v\right)\cap\Im(u,v)\neq\varnothing$\ and\ $F_{12}(\overline u,\overline v)>0$\ for\ any\ $(\overline u,\overline v)\in I_u\times I_v$},
\end{split}\end{equation}
then~$(u, v)$ has one-dimensional symmetry, and $\omega_{u}=\omega_{v}$.
\end{theorem}

\begin{proof}
By the assumptions of Theorem \ref{Tn2}, $|\nabla u|$ and $|\nabla v|$ are taken to be bounded.
Moreover, thanks to either (A1) or (A2) and either (B1) or (B2), the maps
$$
  t \mapsto t^2 \lambda_1(t) +t^2 \lambda_2(t), \qquad t\mapsto t^2 \gamma_1(t) + t^2\gamma_2(t)
$$
belong to $L^{\infty}_{loc}\left([0, +\infty)\right)$.
Therefore, we have that
$$
  |A(\nabla u(x))| \, |\nabla u(x)|^2 + |B(\nabla v(x))|\, |\nabla v(x)|^2 \leq C,
$$
for some positive constant $C$.

Then,
\begin{eqnarray*}
&& \frac1{\log^2 R} \int_{B_R\setminus B_{\sqrt R}} \frac{|A(\nabla u(x))|\, |\nabla u(x)|^2 + |B(\nabla v(x))|\,|\nabla v(x)|^2}{|x|^2}\, dx  \\
&\leq&  \frac{C}{\log^2 R} \int_{\sqrt{R}}^R  \frac{1}{r}\, dr = \frac{C}{\log R}.
\end{eqnarray*}
Therefore, letting~$R\rightarrow +\infty$, we have that the condition~\eqref{EE} is satisfied.
Hence, by Theorem~\ref{T1D}, we obtain the desired result.
\end{proof}

Notice that, as a particular case of \eqref{system}, we can consider the following system, which
arises in phase separation for multiple states Bose-Einstein condensates:
\begin{eqnarray}\label{phase}
\left\{
\begin{array}{ll}
 \Delta u  = uv^2,         \\
  \Delta v = vu^2, \\
  u,v>0.
 \end{array}
\right.
\end{eqnarray}
In fact, in this case, the operators in \eqref{system} reduce to the standard Laplacian and
$F(u,v)=\frac{1}{2}u^2 v^2$.
Under the assumptions of Theorem \ref{Tn2} (notice that  $F_{12}(u,v)=2uv>0$),
one has that the monotone solutions of \eqref{phase} have one-dimensional symmetry.
This result has been proved in \cite{BLWZ}.

\section*{Acknowledgments} The author wants to thank \emph{Enrico Valdinoci}, 
\emph{Andrea Pinamonti} and the anonymous Referee for their useful comments and suggestions.

\medskip
 Received July 2012; revised September 2012.
\medskip

\end{document}